\documentclass[a4paper, 11pt, reqno]{amsart}%

\usepackage{amssymb}      
\usepackage{amsmath}  
\usepackage{amsthm}
\usepackage{amsfonts}  
   
\usepackage{graphicx}

\theoremstyle{plain}
\newtheorem{theorem}{Theorem}[section]

\newtheorem{cor}[theorem]{Corollary}

\newtheorem{lem}[theorem]{Lemma}

\newtheorem{prop}[theorem]{Proposition}

\theoremstyle{definition}

\newtheorem{rem}[theorem]{Remark}
\numberwithin{equation}{section}
\numberwithin{theorem}{section}

\allowdisplaybreaks[4] 

\begin{document}

\title[Logistic elliptic equation with coastal fishery harvesting]{Logistic elliptic equation with a nonlinear boundary condition arising from coastal fishery harvesting}

\author{Kenichiro Umezu} 

\address{Department of Mathematics, Faculty of Education, Ibaraki University, Mito 310-8512, Japan}

\email{\tt kenichiro.umezu.math@vc.ibaraki.ac.jp} 

\subjclass{35J20, 35J25, 35J65, 35B32, 35P30, 92D40}  

\keywords{logistic elliptic equation, nonlinear boundary condition, positive solution set, global subcontinuum, variational approach, Nehari manifold, bifurcation analysis, sub and supersolutions, coastal fishery harvesting} 


\date{\today}

    \maketitle

\begin{abstract}
Let $0<q<1<p$. In this study, we investigate positive solutions of the logistic elliptic equation $-\Delta u = u(1-u^{p-1})$ in a smooth bounded domain $\Omega$ of $\mathbb{R}^N$, $N\geq1$, with the nonlinear boundary condition $\frac{\partial u}{\partial \nu}=-\lambda u^q$ on $\partial\Omega$. This nonlinear boundary condition arises from coastal fishery harvesting. When $p>1$ is subcritical, we prove that in the case of $\lambda_{\Omega}>1$, there exist at least two positive solutions for $\lambda>0$ sufficiently small but no positive solutions for $\lambda>0$ large enough. In the case of $\lambda_{\Omega}<1$, there exists at least one positive solution for every $\lambda>0$. Here, $\lambda_{\Omega}>0$ is the smallest eigenvalue of $-\Delta$ under the Dirichlet boundary condition. 
An interpretation of our main results from an ecological viewpoint is presented. 
\end{abstract}

\section{Introduction and main results} 

Let $\Omega\subset \mathbb{R}^N$, $N\geq1$, 
be a bounded domain with smooth boundary $\partial\Omega$. In this paper, we consider the following logistic elliptic equation with the nonlinear boundary condition arising from coastal fishery harvesting (\cite{GUU19}): 
\begin{align} \label{p}
\begin{cases}
-\Delta u = u(1-u^{p-1}) & \mbox{ in } \Omega, \\
u\geq0 & \mbox{ in } \Omega, \\ 
\frac{\partial u}{\partial \nu}  = -\lambda u^q & \mbox{ on } \partial\Omega. 
\end{cases}
\end{align}
Here, 
\begin{itemize}\setlength{\itemsep}{0.1cm} 
\item $\Delta = \sum_{i=1}^N \frac{\partial^2}{\partial x_i^2}$ is the usual Laplacian in $\mathbb{R}^N$;
\item $\lambda\geq0$ is a parameter; 
\item $0<q<1<p$, and $p<\frac{N+2}{N-2}$ if $N>2$;  
\item $\nu$ is the unit outer normal to $\partial\Omega$.
\end{itemize} 

A nonnegative function $u\in H^1(\Omega)$ is called a \textit{nonnegative (weak) solution} of \eqref{p} if $u$ satisfies  
\begin{align} \label{def}
\int_{\Omega} (\nabla u \nabla \varphi -u\varphi + u^p \varphi) + \lambda \int_{\partial\Omega} u^{q} \varphi = 0, \quad \varphi \in H^1(\Omega)
\end{align}
(we may regard $(\lambda,u)$ as a nonnegative solution of \eqref{p}). 
A nonnegative solution of \eqref{p} belongs to the Sobolev space $W^{1,r}(\Omega)$ for $r>N$ (and consequently to the H\"older space $C^{\theta}(\overline{\Omega})$ for $\theta \in (0,1)$) (\cite[Theorem 2.2]{Ro2005}), and moreover, belongs to $C^{2+\theta}(\Omega)$ for $\theta \in (0,1)$ by Schauder's interior estimate. 
A nonnegative solution $u$ of \eqref{p} is called \textit{positive} if 
$u>0$ in $\Omega$. The strong maximum principle (\cite{PW67}) is applicable to show that a nontrivial nonnegative solution of \eqref{p} is positive. 
However, in terms of the {\it sublinearity} with $0<q<1$, we do not know if a positive solution of \eqref{p} achieves the positivity on the entire boundary $\partial\Omega$. In fact, Hopf's boundary point lemma (\cite{PW67}) is not necessarily applicable for a positive solution of \eqref{p} because 
the function $t\mapsto -t^q$ ($t\geq 0$) does not satisfy the slope condition, i.e., the slope of this function is not bounded below. 
Moreover, it does not necessarily hold that $u\in C^2(\overline{\Omega})$ for a positive solution $u$ of \eqref{p}. As a matter of fact, we prove that $u\not\in C^{1}(\Omega\cup \{ x_0\})$ if $u$ is a positive solution of \eqref{p} that takes a zero value at $x_0\in \partial\Omega$ (\cite[Theorem 2]{Va84}).  

When $p=2$, the unknown function $u>0$ in $\Omega$ ecologically represents the biomass of fish that inhabit a \textit{lake} $\Omega$, obeying the logistic law, where the nonlinear boundary condition $\frac{\partial u}{\partial \nu}=-\lambda u^{q}$ on $\partial\Omega$ with $0<q<1$ means fishery harvesting with the harvesting effort $\lambda$ on the \textit{lake coast} $\partial\Omega$ (\cite[Subsection 2.1]{GUU19}).  

The sublinear nonlinearity $u^q$ with $0<q<1$ appearing in \eqref{p} induces the absorption effect on $\partial\Omega$. Sublinear boundary conditions were explored in \cite{GA04, imcom2007, Wu07, GM08, BW08, RQU2015, RQU2016, RQU2016b, RQU2019}. The case of sublinear incoming flux on $\partial\Omega$, the mixed case of sublinear absorption and incoming flux on $\partial\Omega$, and the sublinear term $u^q$ multiplied by indefinite weight were studied in \cite{GA04, Wu07, GM08, RQU2016b}, in \cite{imcom2007}, and in \cite{BW08, RQU2015, RQU2016, RQU2019}, respectively. 
It should be emphasized that $f(t)=t(1-t^{p-1})$ for $t\in \mathbb{R}$ is concave, whereas $g(t)=-\lambda t^q$ is convex for 
$t\geq0$. The combined nonlinearity can achieve multiplicity of positive solutions of \eqref{p} in certain cases. The study of nonlinear problems with a {\it concave-convex} nature originates from Ambrosetti, Brezis, and Cerami \cite{ABC94} who consider the positive solutions of the problem $-\Delta u =\lambda u^q + u^p$ in $\Omega$ with $0<q<1<p$ under the Dirichlet boundary condition. Studies of concave-convex problems were presented in \cite{FGU03, DS2003, FGU06, Wu06, GM08, K13, RQU2017, RQU2017b, KRQU2019}. 

It is well known that the Neumann logistic problem, \eqref{p} with $\lambda=0$:
\begin{align}  \label{Neup}
\begin{cases}
-\Delta u = u(1-u^{p-1}) & \mbox{ in } \Omega, \\
u\geq0 & \mbox{ in } \Omega, \\ 
\frac{\partial u}{\partial \nu}  = 0 & \mbox{ on } \partial\Omega. 
\end{cases}
\end{align}
has exactly two nonnegative solutions, $u\equiv 0$ and $1$. For $\lambda>0$, it is easy to see that $u\equiv c\geq1$ is a supersolution of \eqref{p}. By a comparison argument, a positive solution $u$ of \eqref{p} satisfies that $u<1$ in $\overline{\Omega}$ (Proposition \ref{prop2.1}). Let $\lambda_{\Omega}>0$ be the smallest eigenvalue of the Dirichlet eigenvalue problem
\begin{align*} 
\begin{cases}
-\Delta \phi = \lambda \phi & \mbox{ in } \Omega, \\
\phi = 0 & \mbox{ on } \partial \Omega. 
\end{cases}    
\end{align*}
We refer to $\phi_{\Omega}\in C^{2+\theta}(\overline{\Omega})$ as a positive eigenfunction associated with $\lambda_{\Omega}$, and it is well known that $\lambda_{\Omega}$ is simple, $\phi_{\Omega}>0$ in $\Omega$, and $\frac{\partial \phi_{\Omega}}{\partial \nu}<0$ on $\partial \Omega$. The smallest eigenvalue $\lambda_{\Omega}$ is characterized by the variational formula
\begin{align} \label{lamOcha}
\lambda_{\Omega} = \inf\left\{ \int_{\Omega}|\nabla u|^2 : u\in H^1_0(\Omega), \ \int_{\Omega}u^2=1 \right\}. 
\end{align}
The sign of $\lambda_{\Omega}-1$ plays an essential role in determining the positive solution set $\{ (\lambda, u)\}$ of \eqref{p}. In the case of $\lambda_{\Omega}<1$, 
$\varepsilon \phi_{\Omega}$ is a subsolution of \eqref{p} for any $\lambda>0$, provided that $\varepsilon>0$ is sufficiently small; therefore, the sub and supersolution method ensures that problem \eqref{p} has at least one positive solution $u$ for every $\lambda>0$ such that $\varepsilon \phi_{\Omega}\leq u \leq1$ in $\overline{\Omega}$ (Proposition \ref{prop:subsuper}). 
However, the situation is different in the case of $\lambda_{\Omega}>1$. In this case, the variational characterization \eqref{lamOcha} provides that $\int_{\Omega}(|\nabla u|^2-u^2)\geq0$ for $u\in H^1_0(\Omega)$, where the equality holds only when $u=0$. 
This enables us to prove that a positive solution $u$ of \eqref{p} satisfies that $\| u\|_{L^p(\Omega)}\geq C\lambda^{\frac{1}{1-q}}$ as $\lambda\to \infty$. 
Combining this lower bound and the upper bound $u<1$ in $\overline{\Omega}$ for the positive solution $u$ of \eqref{p} shows that problem \eqref{p} has no positive solution for any $\lambda>0$ large enough (Proposition \ref{prop:boundlam}).

Now, we present our main results in this paper. 
The first main result presents, unconditionally for every $\lambda_\Omega > 0$, 
the uniform upper bound for the positive solutions of \eqref{p} and their positivity on $\partial\Omega$, and the existence and uniqueness of a smooth positive solution curve $\{ (\lambda, u_{1,\lambda})\}$ of \eqref{p} emanating from $(\lambda,u)=(0,1)$. 
The definitions of the asymptotic stability and the instability for positive solutions $(\lambda,u)$ of \eqref{p} with the condition that $u>0$ in $\overline{\Omega}$ are referred to in \eqref{linearp} below. 
The existence part of Theorem \ref{th1} holds for all $p>1$. 

\begin{theorem} \label{th1} 
Let $u$ be a positive solution of \eqref{p} for $\lambda>0$. Then, 
$u<1$ in $\overline{\Omega}$ and $u>0$ on $\Gamma$ with some $\Gamma\subset \partial\Omega$ satisfying that $|\Gamma|>0$. 
Conversely, problem \eqref{p} has a smooth positive {\rm solution curve} $\{ (\lambda, u_{1,\lambda}) : 0\leq  \lambda<\overline{\lambda} \}$ with some $\overline{\lambda}>0$ 
in $\mathbb{R}\times C^{2+\theta}(\overline{\Omega})$, $\theta \in (0,1)$, such that 
$u_{1,0}=1$, $u_{1,\lambda}$ is asymptotically stable. Moreover, the positive solution set $\{ (\lambda,u)\}$ of \eqref{p} forms the smooth positive solution curve in a neighborhood of $(\lambda,u)=(0,1)$. 
\end{theorem}

If $\lambda_{\Omega}<1$, then it is well known (\cite{BO86}) that the logistic Dirichlet problem 
\begin{align} \label{pDl}
\begin{cases}
-\Delta u = u(1-u^{p-1}) & \mbox{ in } \Omega, \\
u\geq0 & \mbox{ in } \Omega, \\ 
u=0 & \mbox{ on } \partial\Omega, 
\end{cases}
\end{align}
has a unique positive solution $u_{\mathcal{D}}\in C^{2+\theta}(\overline{\Omega})$, such that $u_{\mathcal{D}}>0$ in $\Omega$, and $\frac{\partial u_{\mathcal{D}}}{\partial \nu}<0$ on $\partial\Omega$. 
The second main result is in the case of $\lambda_{\Omega}<1$, which provides a global existence result for the positive solutions of \eqref{p}. The existence part of Theorem \ref{th2} holds for all $p>1$. 
\begin{theorem} \label{th2}
Assume that $\lambda_{\Omega}<1$. Then, problem \eqref{p} has at least one positive solution $u_{\lambda}\in C^{2+\theta}(\overline{\Omega})$, $\theta \in (0,1)$, for each $\lambda>0$ such that $u_{\lambda}>0$ in $\overline{\Omega}$. Furthermore, it holds that 
\begin{enumerate} \setlength{\itemsep}{0.2cm} 
\item $u_\lambda$ is unique for $\lambda>0$ small (i.e., $u_{\lambda}$ coincides with  $u_{1,\lambda}$ given by Theorem \ref{th1}) ;  
\item $u_{n} \rightarrow u_{\mathcal{D}}$ in $H^1(\Omega)$ for a positive solution $u_{n}$ of \eqref{p} with $\lambda= \lambda_n \rightarrow \infty$, 
\end{enumerate}
see Figure \ref{fh01a}. 
\end{theorem} 

\begin{rem}
\strut\begin{enumerate} \setlength{\itemsep}{0.2cm} 

\item From Theorem \ref{th1}, we observe that $u_{\lambda}<1$ in $\overline{\Omega}$. 

\item No bifurcation from the trivial line $\{ (\lambda,0) : \lambda \geq 0\}$ for positive solutions of \eqref{p} occurs because of assertion (i) and Proposition \ref{prop:nobif}. 

\end{enumerate} 
\end{rem}

The third main result is in the case of $\lambda_{\Omega}>1$, which provides a local multiplicity result for the positive solutions of \eqref{p}. 
\begin{theorem} \label{th3} 
Suppose that $\lambda_{\Omega}>1$. Then, problem \eqref{p} has at least {\rm two} positive solutions $U_{1,\lambda}, U_{2,\lambda}$ 
for $\lambda\in (0,\hat{\lambda})$ with some  $\hat{\lambda}>0$ such that 
$U_{1,\lambda}\rightarrow 1$ in $C^{2+\theta}(\overline{\Omega})$, 
$\theta \in (0,1)$, and $U_{2,\lambda} \rightarrow 0$ in $H^1(\Omega)$ as $\lambda \to 0^{+}$ (implying that $U_{1,\lambda}$ coincides with $u_{1,\lambda}$ given by Theorem \ref{th1}).  Additionally, the following two assertions hold: 
\begin{enumerate}\setlength{\itemsep}{0.2cm} 
\item Given $U_{2,\lambda_{n}}$ with $\lambda_n \rightarrow 0^{+}$, we have that up to a subsequence,  $v_n= \lambda_n^{-\frac{1}{1-q}}U_{2,\lambda_n} \rightarrow v_0$ in $H^1(\Omega)$. Here, $v_0$ is a positive solution of the problem 
\begin{align} \label{lp}
\begin{cases}
-\Delta v = v & \mbox{ in } \Omega, \\
\frac{\partial v}{\partial \nu} = -v^q & \mbox{ on } \partial\Omega,  
\end{cases}    
\end{align}
which admits that $v_0>0$ on $\Gamma$ with some $\Gamma\subset \partial\Omega$ satisfying that $|\Gamma|>0$. 
\item $U_{2,\lambda}$ is unstable, provided that it is positive in $\overline{\Omega}$. 
\end{enumerate}
Meanwhile, there is {\rm no} positive solution of \eqref{p} for $\lambda > 0$ large enough.  
\end{theorem}

Furthermore, problem \eqref{p} possesses a bounded subcontinuum $\{ (\lambda,u)\}$ in $[0,\infty)\times C(\overline{\Omega})$ of nonnegative solutions, joining $(0,0)$ to $(0,1)$, 
meaning that $(\lambda, U_{1,\lambda})$ and $(\lambda, U_{2,\lambda})$ could be linked with a bounded, closed connected subset of nonnegative solutions $(\lambda,u)$ of \eqref{p}.  

The fourth main result is the following: 
\begin{theorem} \label{th4}
Suppose that $\lambda_{\Omega}>1$. Then, problem \eqref{p} possesses a bounded {\rm subcontinuum} (i.e., nonempty, closed, and connected subset) $\mathcal{C}_0=\{ (\lambda, u)\}$ of nonnegative solutions in $[0,\infty)\times C(\overline{\Omega})$ such that: 
\begin{enumerate} \setlength{\itemsep}{0.2cm} 

\item $(0,0), (0,1) \in \mathcal{C}_{0}$;

\item $\mathcal{C}_0 \cap \{ (\lambda,0) \cup (0,u)\}=\{(0,0), (0,1)\}$, which implies that $\mathcal{C}_{0}\setminus \{(0,0) \}$ consists of positive solutions of \eqref{p};

\item $\{ (\lambda, u_{1,\lambda}) : 0\leq  \lambda<\overline{\lambda} \} \subset \mathcal{C}_{0}$,

\end{enumerate}
see Figure \ref{fh01b}.   
\end{theorem}

\begin{rem}
\strut
\begin{enumerate} \setlength{\itemsep}{0.2cm} 

\item From Theorem \ref{th1}, we observe that if $(\lambda, u)\in \mathcal{C}_{0}\setminus \{ (0,0)\}$ for $\lambda>0$, then $u<1$ in $\overline{\Omega}$. 
\item No bifurcation from the trivial line $\{ (\lambda, 0) : \lambda > 0\}$ for positive solutions of \eqref{p} occurs because of Proposition \ref{prop:nobif}. 
    
\item Actually, Theorem \ref{th4} provides that problem \eqref{p} has at least two positive solutions for $\lambda > 0$ small as stated in Theorem \ref{th3}. 
It is worthwhile mentioning that Theorem \ref{th3} presents the asymptotic profile of the second positive solution $U_{2,\lambda}$ as $\lambda \to 0^{+}$. 
Moreover, from the proof of Theorem \ref{th3}, the parameter range of $\lambda$ for which the multiplicity of positive solutions holds is estimated with $p$ and $q$ (\eqref{lamast} and \eqref{muastdef}). 

\end{enumerate} 
\end{rem}

From the variational viewpoint, the functional \eqref{def:J} associated with \eqref{p} and introduced in Section \ref{sec:multi} is {\it coercive}, and is bounded from below (Lemma \ref{lem:Jcoer}). Therefore, problem \eqref{p} has a least energy solution for \textit{every} $\lambda>0$ with respect to the functional, which is nonnegative in $\overline{\Omega}$. In the case of $\lambda_{\Omega}>1$, problem \eqref{p} has no positive solution for $\lambda>0$ large (Theorem \ref{th3}), which implies that the least energy solution is {\it zero} for such $\lambda$. It should be noted that the nonnegative solutions of \eqref{p} are the steady state solutions of the nonlinear initial boundary value problem
\begin{align} \label{pp}
\left\{ \begin{array}{ll}
\dfrac{\partial u}{\partial t} = \Delta u + u(1-u^{p-1})  
& \mbox{ in } \ (0,\infty)\times \Omega, \vspace{3pt}\\
\dfrac{\partial u}{\partial \nu} = - \lambda u^{q} & \mbox{ on } (0,\infty)\times \partial\Omega, \\
u(0,x)=u_0(x)\geq 0 & \mbox{ for } x\in \overline{\Omega}. 
\end{array}\right. 
\end{align}
In the case of $\lambda_{\Omega}>1$, it would be expected that the global nonnegative solution of \eqref{pp} vanishes as $t\to \infty$ for $\lambda>0$ large enough, and that the positive steady state solution {\it jumps down} to zero at $\lambda=\lambda^{\ast}:=\sup\{ \lambda>0 : (\lambda,u)\in \mathcal{C}_{0}\}$ in terms of the existence of a $\supset$ shaped subcontinuum of positive solutions (Theorem \ref{th4}). From the viewpoint of fishery harvesting, we could infer that by overfishing, the collapse of the stock for fish occurs as the harvesting effort $\lambda$ moves beyond $\lambda^{\ast}$ (Figure \ref{fh01b}).  
However, the extinction wouldn't occur in the case of  $\lambda_{\Omega}<1$ (Figure \ref{fh01a}). The positive solution $u_{\lambda}$ by Theorem 1.2 is weakly stable in the sense of Amann \cite{Am76}, because it is constructed via sub and supersolutions (Proposition \ref{prop:subsuper}). Therefore, it would be expected that the global nonnegative solution of \eqref{pp} converges to a positive steady state solution for every $\lambda>0$. Nevertheless, Theorem \ref{th2} (ii) shows that the positive steady state solution remains positive in $\Omega$ 
but vanishes on $\partial\Omega$ as $\lambda\to \infty$. From the viewpoint of fishery harvesting, it can be inferred that the population of fish becomes extinct on the coast $\partial \Omega$ eventually. Mathematically, we know that the larger the size of the domain $\Omega$, the smaller $\lambda_{\Omega}$ is. Therefore, our ecological interpretation would be consistent with this mathematical fact.

    \begin{figure}[!htb]
    \centering 
    \includegraphics[scale=0.19]{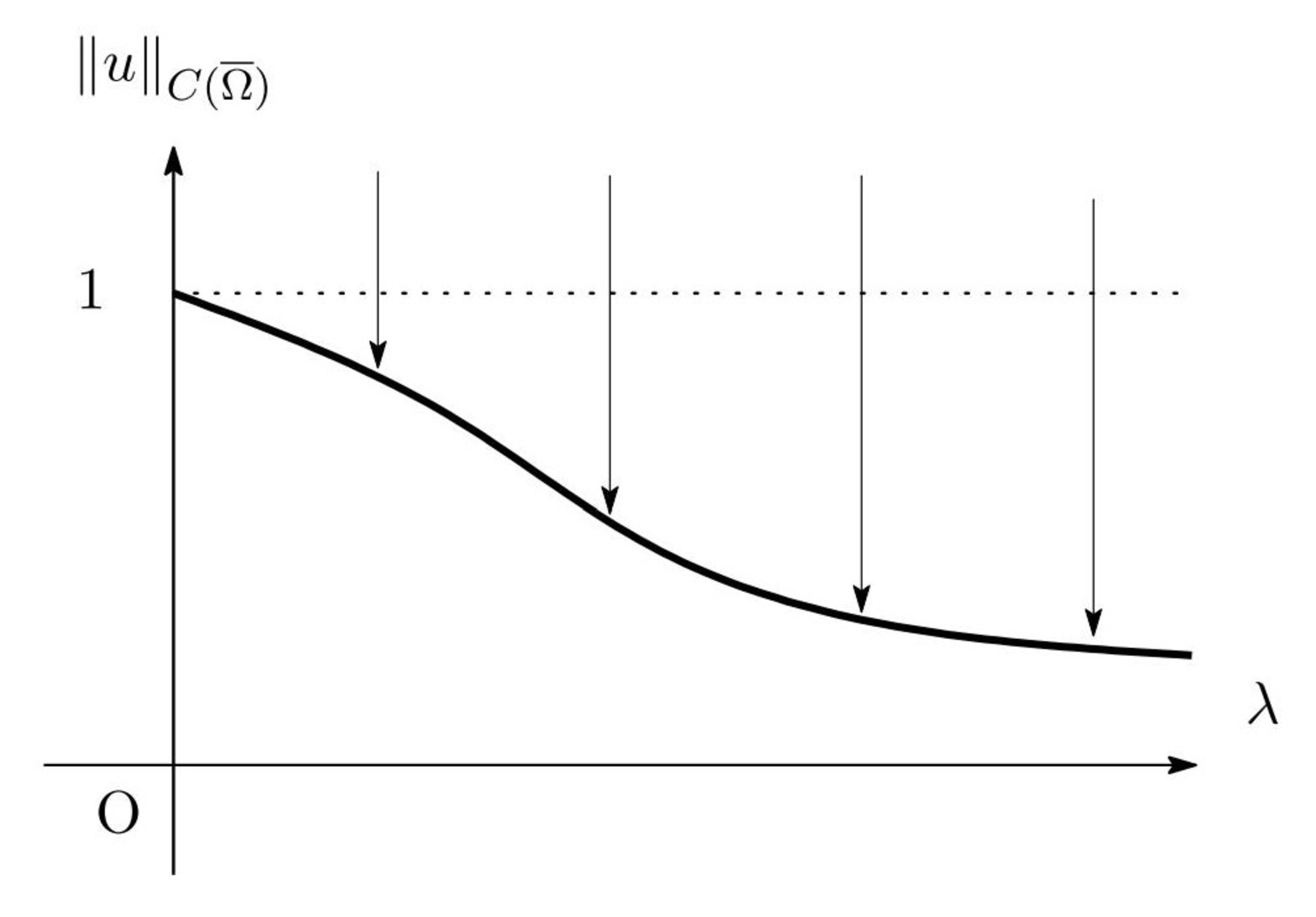} 
		  \caption{Possible positive solution set of \eqref{p} in the case of $\lambda_{\Omega}<1$. 
		  }
		\label{fh01a} 
    \end{figure} 
%
    \begin{figure}[!htb]
    \centering 
    \includegraphics[scale=0.19]{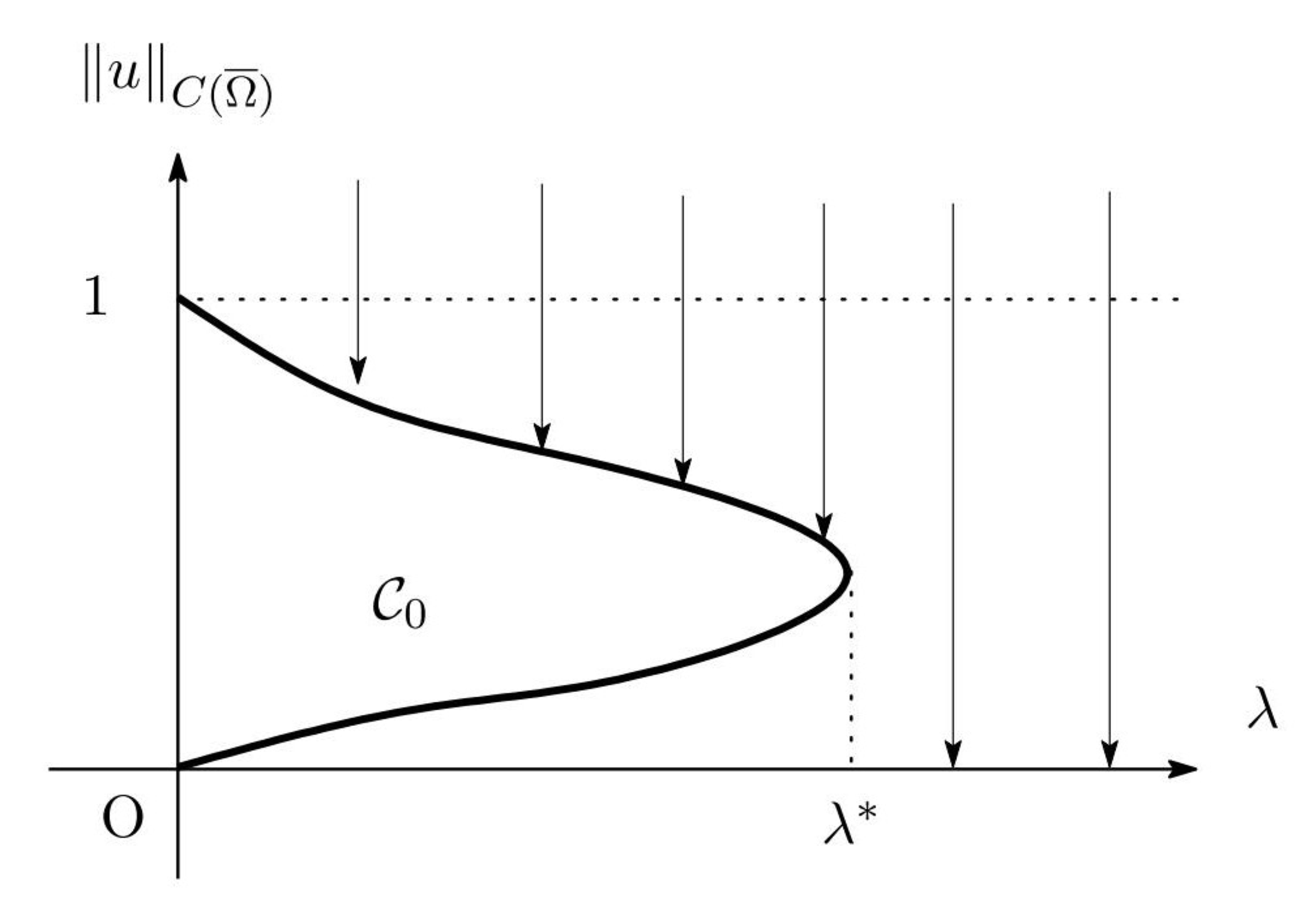} 
		  \caption{Possible positive solution set of \eqref{p} in the case of $\lambda_{\Omega}>1$.
		  }
	\label{fh01b} 
    \end{figure} 

In the one-dimensional case of the limiting problem \eqref{lp} 
\begin{align} \label{lp1}
\begin{cases}
-v^{\prime\prime}(x) = v(x), & x\in (a,b), \\
-v^{\prime}(a)=-v(a)^{q}, & \\ 
v'(b)=-v(b)^{q}, &  \\ 
\end{cases}    
\end{align} 
with $a<b$, a positive solution $v$ of \eqref{lp1} is positive in $[a,b]$. The equation $-v^{\prime\prime}=v$ admits the general solution 
$v_c(x)=C_1 \sin (x+C_2)\in C^{\infty}(\mathbb{R})$ with $C_1, C_2 \in \mathbb{R}$. When $C_1\neq 0$, $v_c^{\prime}(x_{0})\neq 0$ if $v_c(x_{0})=0$, and thus, the boundary condition implies that $v_c(a)v_c(b)>0$ if $v_c>0$ in $(a,b)$. Since we have Theorem \ref{th3} (i) with $v_0>0$ in $[a,b]$ accordingly, Theorem \ref{th3} (i) and (ii) provide that $U_{2,\lambda}>0$ in $[a,b]$ is unstable for $\lambda>0$ sufficiently small (using Lemma \ref{lem:compact}). 

In this paper, we do not discuss if problem \eqref{p} admits a positive solution $u$ satisfying that $u=0$ somewhere on $\partial\Omega$. It is an interesting open question whether such \textit{degenerate} positive solutions exist or not.

To conclude the Introduction, we refer to the case $\lambda<0$. For $\lambda < 0$, problem \eqref{p} possesses a {\it concave-concave} nature. Thus, the positive solution set $\{ (\lambda, u)\}$ is not complicated. Indeed, 
there exists a unique positive solution of \eqref{p} for $\lambda<0$, and it is asymptotically stable. The proofs are carried out in the same spirit of that for \cite[Theorem 1.2]{GM08}. 
The asymptotic profile of the unique positive solution as $\lambda \to -\infty$ is still an open question.

\begin{description} 
\item[Notation] 
$\| \cdot\|$ denotes the usual norm of $H^1(\Omega)$. 
$u_n \rightharpoonup u_\infty$ means that $u_n$ converges weakly to $u_\infty$ in $H^1(\Omega)$. $u\neq 0$ means that $u\in H^1(\Omega)\setminus \{ 0\}$. $H^1_0(\Omega)=\{ u\in H^1(\Omega) : u=0 \mbox{ on } \partial\Omega\}$. $\int_{\Omega} f dx$ for $f \in L^1(\Omega)$ and $\int_{\partial\Omega}g d\sigma$ for $g\in L^1(\partial\Omega)$ are written simply as 
$\int_{\Omega}f$ and $\int_{\partial\Omega}g$, respectively. 
$|\cdot|$ represents the Lebesgue measure in $\Omega$ and the surface measure on $\partial\Omega$ both. 
\end{description}


The remainder of this paper is organized as follows. 

In Section \ref{sec:sub}, we prove the local existence result for a positive solution of \eqref{p} emanating from $(\lambda,u)=(0,1)$, using the implicit function theorem. 
Additionally, we conduct stability analysis for the positive solution. Moreover, in the case when $\lambda_{\Omega}<1$, 
we establish the global existence result for a positive solution of \eqref{p} for every $\lambda>0$, applying the sub and supersolution method. We present a proof for Theorem \ref{th1} in this section. 

In Section \ref{sec:multi}, we establish the local multiplicity result for positive solutions of \eqref{p}, using the variational method on the Nehari manifold associated with \eqref{p}. 

In Section \ref{sec:asympt}, using the variational technique, we evaluate the asymptotic profile for positive solutions of \eqref{p} as $\lambda\to 0^{+}$ and $\lambda \to \infty$. We present proofs for Theorems \ref{th2} and \ref{th3} in this section.  

In Section \ref{sec:topol}, with the aid of a topological argument by Whyburn \cite{Wh64}, we prove Theorem \ref{th4} via the bifurcation technique. Our bifurcation approach to \eqref{p} from zero solutions is {\it non standard}, meaning that the well known local and global bifurcation results from simple eigenvalues by Crandall and Rabinowitz \cite{CR71}, Rabinowitz \cite{Ra71}, or L\'opez-G\'omez \cite{LGbook} are not directly applied because the function $t^q$ ($t\geq0$) with $0<q<1$ is not right differentiable at $t=0$. We perform a regularization scheme for \eqref{p} to overcome this difficulty.


\section{Existence of positive solutions via the implicit function theorem and sub and supersolutions} 

\label{sec:sub}

First, we establish the {\it a priori} upper bound of the uniform norm on $\overline{\Omega}$ and positivity on $\partial\Omega$ for the positive solutions of \eqref{p}. 
\begin{prop} \label{prop2.1}
Let $u$ be a positive solution of \eqref{p}. 
Then, the following two assertions hold:
\begin{enumerate} \setlength{\itemsep}{0.2cm} 
\item $u<1$ in $\overline{\Omega}$; 
\item $u>0$ on some $\Gamma\subset \partial\Omega$ satisfying that $|\Gamma|>0$. 
\end{enumerate}
\end{prop}

\begin{proof}
(i) Assume by contradiction that $M:= \max_{\overline{\Omega}}u\geq1$ for a positive solution $u$ of \eqref{p}. Additionally assume that $u(x_0)=M$ for $x_0\in \Omega$. Choose $K>0$ such that $Kt + f(t)$ is increasing for $t\in[0, M]$, where $f(t)=t(1-t^{p-1})$. Using the fact that $-\Delta M = 0\geq f(M)$, we deduce that 
\begin{align*}
(-\Delta + K)(M-u)\geq (KM+f(M))-(Ku+f(u))\geq 0 \ \mbox{ in } 
\Omega. 
\end{align*}
Since $u \in C^2(\Omega)\cap C(\overline{\Omega})$, the strong maximum principle applies, and then, $M-u$ is identically equal to zero, i.e., $u\equiv M$, which is contradictory for the nonlinear boundary condition. Hence, we obtain that $x_0 \in \partial\Omega$. However, $u \in C^1$ in a neighborhood of $x_0$ by the bootstrap argument based on the fact that $u \in W^{1,r}(\Omega)$ with $r>N$, thus, Hopf's boundary point lemma applies at $x_0$. We then arrive at the contradiction that 
\begin{align*}
0>-\lambda u^{q}(x_0) =\frac{\partial u}{\partial \nu}(x_0)>0. 
\end{align*} 

(ii) Since $u$ is a positive solution of \eqref{p}, it follows from \eqref{def} with $\varphi=1$ that 
\begin{align*}
\int_{\Omega} u(1-u^{p-1}) = \lambda \int_{\partial\Omega} u^q. 
\end{align*}
Assertion (i) shows that $\int_{\partial\Omega}u^q>0$, as desired. \end{proof}


Set $U=\left\{ (\lambda, u) \in \mathbb{R}\times C^{2+\theta}(\overline{\Omega}): u>0 \ \mbox{ in } \overline{\Omega} \right\}$ with $\theta \in (0,1)$, and define the nonlinear mapping 
\begin{align} \label{mapF}
F: U \longrightarrow C^{\theta}(\overline{\Omega})\times 
C^{1+\theta}(\partial\Omega); \quad (\lambda,u) \mapsto 
\left( -\Delta u - u(1-u^{p-1}), \, \frac{\partial u}{\partial \nu} + \lambda u^q \right). 
\end{align}
We then deduce that $(\lambda,u)$ is a positive solution of \eqref{p} in $U$ if and only if $F(\lambda,u)=0$. Consider the following linearized eigenvalue problem of $F(\lambda,u)=0$ at $(\lambda,u)$: 
\begin{align} \label{linearp}
\begin{cases}
-\Delta \varphi = f'(u)\varphi + \gamma \varphi & \mbox{ in } \Omega, \\
\frac{\partial \varphi}{\partial \nu}=-\lambda g'(u)\varphi + \gamma \varphi & 
\mbox{ on } \partial\Omega, 
\end{cases}    
\end{align}
where $f(t)=t(1-t^{p-1})$ and $g(t)=t^q$. We refer to $\gamma_1=\gamma_1(\lambda,u)$ and  $\varphi_1=\varphi_1(\lambda,u)$ as the smallest eigenvalue and a corresponding positive eigenfunction associated with $\gamma_1$, respectively. It is well known that $\gamma_1$ is simple, and $\varphi_1>0$ in $\overline{\Omega}$. A positive solution $(\lambda, u)$ of \eqref{p} satisfying that $u>0$ in $\overline{\Omega}$ is said to be {\it asymptotically stable}, {\it unstable}, and {\it weakly stable} if $\gamma_1(\lambda,u)>0$, $\gamma_1(\lambda,u)<0$, and $\gamma_1(\lambda,u)\geq0$, respectively. 

Second, the implicit function theorem proves the following existence and uniqueness result for the positive solutions of \eqref{p}.  
\begin{prop} \label{prop:bif}
Problem \eqref{p} has a smooth positive solution curve $\{(\lambda, u_{1,\lambda}): |\lambda|< \overline{\lambda} \}$ with a small $\overline{\lambda}>0$ in $\mathbb{R}\times C^{2+\theta}(\overline{\Omega})$ for $\theta\in (0,1)$, which satisfies that $u_{1,0}=1$, and $0<u_{1,\lambda}<1$ in $\overline{\Omega}$ for $\lambda>0$. Moreover, there is a neighborhood of $(\lambda,u)=(0,1)$ such that if $u$ is a positive solution of \eqref{p} in the neighborhood, then $u=u_{1,\lambda}$ for some $\lambda > 0$.  
\end{prop}

\begin{proof} \label{prop3.1}
Choose an open neighborhood $U$ of $(\lambda,u)=(0,1)$ in $\mathbb{R}\times C^{2+\theta}(\overline{\Omega})$, and consider \eqref{mapF} and \eqref{linearp}. We then observe that $\gamma_1(0,1)>0$. 
Therefore, the mapping $(-\Delta - f'(1), \frac{\partial}{\partial \nu})$ is a homeomorphism, and the implicit function theorem is applicable to deduce all the assertions except that of $u_{1,\lambda}<1$ in $\overline{\Omega}$. The assertion that $u_{1,\lambda}<1$ in $\overline{\Omega}$ follows from Proposition \ref{prop2.1} (i). \end{proof}

Third, we prove the following stability result for $u_{1,\lambda}$.

\begin{prop} \label{prop:stable}
If a positive solution $u$ of \eqref{p} satisfies that 
$u\geq\left(\frac{1-q}{p-q}\right)^{\frac{1}{p-1}}$ in $\overline{\Omega}$, then $u$ is asymptotically stable. Particularly, 
$u_{1,\lambda}$ given by Proposition \ref{prop:bif} is asymptotically stable for $\lambda>0$ close to $0$. 
\end{prop}

\begin{proof}
Let $u$ be a positive solution of \eqref{p} such that $u>0$ in $\overline{\Omega}$. Proposition \ref{prop2.1} (i) shows that $f(u)>0$ in $\overline{\Omega}$. 
Consider the smallest eigenvalue $\gamma_1(\lambda,u)$ of \eqref{linearp} with a corresponding positive eigenfunction $\varphi_1(\lambda,u)$, and observe that 
\begin{align*}
\int_{\Omega} \left( -\Delta \varphi_1 f(u) + \Delta u f'(u)\varphi_1 \right) = \gamma_1 \int_{\Omega} \varphi_1 f(u). 
\end{align*}
By the divergence theorem, we deduce that 
\begin{align*}
&\int_{\Omega} \left( -\Delta \varphi_1 f(u) + \Delta u f'(u)\varphi_1 \right)  \\  
&= -\int_{\Omega} f''(u)|\nabla u|^2\varphi_1 + 
\lambda \int_{\partial\Omega} \frac{d}{du}\! \left( \frac{g(u)}{f(u)}\right) (f(u))^2 \varphi_1 - \gamma_1 \int_{\partial\Omega} \varphi_1 f(u). 
\end{align*}
It follows that 
\begin{align*}
\gamma_1 = \frac{-\int_{\Omega} f''(u)|\nabla u|^2\varphi_1 + 
\lambda \int_{\partial\Omega}\frac{d}{du}\! \left( \frac{g(u)}{f(u)}\right) (f(u))^2 \varphi_1}{\int_{\Omega} f(u)\varphi_1 + \int_{\partial\Omega} f(u)\varphi_1}.    
\end{align*}
By direct computations, 
\begin{align*}
f''(t)=-p(p-1)t^{p-2}, \quad \frac{d}{dt}\! \left( \frac{g(t)}{f(t)}\right)
=\frac{(p-q)\left( t^{p-1}-\frac{1-q}{p-q}\right)}{t^{2-q}(1-t^{p-1})^2}. 
\end{align*}
Hence, if $u\geq \left(\frac{1-q}{p-q}\right)^{\frac{1}{p-1}}$ in $\overline{\Omega}$, then $\gamma_1>0$. The desired conclusion follows. \end{proof}

\begin{rem}
Propositions \ref{prop:bif} and \ref{prop:stable} hold for all $\lambda_\Omega > 0$ and $p>1$.
\end{rem}

We can then prove Theorem \ref{th1}.  
\begin{proof}[Proof of Theorem \ref{th1}]
The first assertion follows from Proposition \ref{prop2.1}. The second assertion follows from Propositions \ref{prop:bif} and  \ref{prop:stable}.  \end{proof}

To conclude this section, we prove the following global existence result for the positive solutions of \eqref{p} in the case when $\lambda_{\Omega}<1$. The subsolution and the supersolution of \eqref{p} are defined as in \cite{Am76N}. 
\begin{prop} \label{prop:subsuper}
Assume that $\lambda_{\Omega}<1$. Then, problem \eqref{p} has at least one positive solution $u_{\lambda}$ for every $\lambda>0$, which satisfies that $u_{\lambda}>0$ in $\overline{\Omega}$. 
\end{prop}

\begin{proof}
It is clear that $u\equiv 1$ is a supersolution of \eqref{p}. 
We construct a certain subsolution of \eqref{p} to ensure the existence of a positive solution of \eqref{p} that is positive on the entire boundary $\partial\Omega$. 

For $\varepsilon, \delta>0$, we set $w = \varepsilon (\phi_{\Omega} + \delta)$ such that $w<1$. We then deduce that 
\begin{align} 
-\Delta w - w(1-w^{p-1}) 
\leq \varepsilon \phi_{\Omega} \{ \lambda_{\Omega} - 1 + \varepsilon^{p-1} (\phi_{\Omega} + \delta)^{p-1}\} \quad\mbox{ in } \Omega, \label{subw1}
\end{align}
\begin{align} 
\frac{\partial w}{\partial \nu}+ \lambda w^q 
\leq \varepsilon \left( -C_{\Omega} + \lambda \varepsilon^{q-1}\delta^{q} \right) \quad\mbox{ on } \partial\Omega,  \label{subw2} 
\end{align}
where $\frac{\partial \phi_{\Omega}}{\partial \nu}\leq-C_{\Omega}$ on $\partial\Omega$ for some $C_{\Omega}>0$. Choose $\delta = \varepsilon^{\tau}$ with $\tau>\frac{1-q}{q}$, and then, assertions \eqref{subw1} and \eqref{subw2} show that $w$ is a subsolution of \eqref{p} if $\varepsilon>0$ is small enough. Therefore, 
the sub and supersolution method \cite[(2.1) Theorem]{Am76N} shows the existence of a positive solution $u$ of \eqref{p} such that $\varepsilon (\phi_{\Omega}+\varepsilon^{\tau}) \leq u \leq 1$ in $\overline{\Omega}$, as desired. \end{proof}

\begin{rem}
\strut
\begin{enumerate} \setlength{\itemsep}{0.2cm} 
\item Proposition \ref{prop:subsuper} holds for all $p>1$. 
\item We observe from Corollary \ref{cor:uniq} below that $u_{\lambda}=u_{1,\lambda}$ for $\lambda>0$ close to $0$, where $u_{1,\lambda}$ is the positive solution of \eqref{p} emanating from $(\lambda,u)=(0,1)$, given by Proposition \ref{prop:bif}. 
\end{enumerate}
\end{rem}


\section{Multiplicity of positive solutions via Nehari manifold}

\label{sec:multi}

In this section, we discuss the existence and multiplicity of positive solutions for \eqref{p} in the case of $\lambda_{\Omega}>1$ by employing the variational method on the Nehari manifold associated with \eqref{p}. 

\subsection{Nehari manifold and fibering map}

We introduce the functional associated with \eqref{p}:
\begin{align} \label{def:J}
J_{\lambda}(u):= \frac{1}{2}E(u)+\frac{1}{p+1}A(u)+\frac{\lambda}{q+1}B(u), \quad u\in H^1(\Omega),     
\end{align}
where 
\begin{align*}
E(u):=\int_{\Omega}(|\nabla u|^2-u^2), \quad 
A(u):=\int_{\Omega}|u|^{p+1}, \quad\mbox{ and } 
B(u):=\int_{\partial\Omega}|u|^{q+1}. 
\end{align*}
In the sequel, we use the notations: 
\begin{align*}
& E^{\pm}=\{ u \in H^1(\Omega) : E(u)\gtrless 0 \}, \\
& A^{+}=\{ u \in H^1(\Omega) : A(u)>0 \}, \\
& B^{+}=\{ u \in H^1(\Omega) : B(u)> 0\}.  
\end{align*}
By a simple calculation, it is easy to verify the following lemma.
\begin{lem} \label{lem:Jcoer}
Let $\lambda>0$. Then, $J_{\lambda}$ is coercive, and is bounded from below in $H^1(\Omega)$. More precisely, we have $C>0$ such that $J_{\lambda}(u)\geq \frac{1}{2}\| u\|^2-C$ for $u\in H^1(\Omega)$. 
Here, $C$ does not depend on $\lambda$. 
\end{lem}

The following three lemmas are used frequently in our arguments. 
\begin{lem} \label{lemD}
Assume that $\lambda_{\Omega}>1$. Then, $E(u)\geq0$ for $u\in H^1_0(\Omega)$, and 
moreover, $E(u)^{\frac{1}{2}}$ and $\| u\|_{H^1_0(\Omega)}$ are equivalent in $H^1_0(\Omega)$.  
\end{lem}

\begin{proof}
Since $\lambda_{\Omega}>1$, it follows from \eqref{lamOcha} 
that $E(u)\geq 0$ for $u\in H^1_0(\Omega)$, and $u\equiv 0$ if $u\in H^1_0(\Omega)$ satisfies that $E(u)=0$. We claim that there exists $C>0$ such that $\| u\|_{H^1_0(\Omega)}^2\leq C E(u)$ for $u\in H^1_0(\Omega)$. We assume to the contrary that $u_n\in H^1_0(\Omega)$, $\| u_n \|_{H^1_0(\Omega)}=1$, but $E(u_n)\rightarrow 0$. Then, up to a subsequence, $u_n \rightharpoonup u_0$ in $H^1_0(\Omega)$ and $u_n \rightarrow u_0$ in $L^2(\Omega)$. This implies that $0\leq E(u_0)\leq \varliminf_{n}E(u_n)\leq \varlimsup_{n}E(u_n)=0$, thus, $u_0=0$. 
Consequently, $E(u_n) \rightarrow 0$, thus, $u_n \rightarrow 0$ in $H^1_0(\Omega)$. This is contradictory for $\| u_n\|_{H^1_0(\Omega)}=1$. \end{proof}

\begin{lem} \label{lem:compact}
Let $u_n$ be a positive solution of \eqref{p} for $\lambda=\lambda_n\geq0$ such that $\lambda_n$ is bounded. Then, $u_n$ is bounded in $W^{1,r}(\Omega)$ for $r>N$ (implying that $u_n$ is bounded in $C^{\theta}(\overline{\Omega})$ for $\theta\in (0,1)$). Furthermore, up to a subsequence, $\lambda_n \rightarrow \lambda_0$, $u_n \rightharpoonup u_0$, and $u_n \rightarrow u_0$ in $C(\overline{\Omega})$.
\end{lem}

\begin{proof}
We may infer that $\lambda_n \rightarrow \lambda_0$. Substituting $\varphi=u_n$ for \eqref{def}, Proposition \ref{prop2.1} (i) shows that 
\begin{align*} 
\int_{\Omega}|\nabla u_n|^2 = \int_{\Omega}\left( u_n^2-A(u_n)-\lambda_n B(u_n) \right) \leq \int_{\Omega}u_n^2\leq |\Omega|. 
\end{align*}
It follows that $u_n$ is bounded in $H^1(\Omega)$, 
implying that up to a subsequence, $u_n \rightharpoonup u_0$, and $u_n \rightarrow u_0$ in $L^{p+1}(\Omega)$ and $L^2(\partial\Omega)$. 
As in the proof of \cite[Theorem 2.2]{Ro2005}, we deduce that $u_n$ is bounded in $W^{1,r}(\Omega)$ for $r>N$. Sobolev's embedding theorem ensures that this is the case in $C^{\theta}(\overline{\Omega})$ with $\theta = 1-\frac{N}{r}$. 
The assertion that $u_n$ has a convergent subsequence in $C(\overline{\Omega})$ follows by the fact that $C^{\theta}(\overline{\Omega})\subset C(\overline{\Omega})$ is compact. The desired conclusion follows. \end{proof}

%
\begin{lem} \label{lem:wlsc}
Let $\{u_n\}\subset H^1(\Omega)$ be such that $E(u_n)\leq0$, 
$u_n \rightharpoonup u_0$, and $u_n \rightarrow u_0$ in $L^2(\Omega)$. 
Then, the following two assertions hold:
\begin{enumerate} \setlength{\itemsep}{0.1cm}
\item If $\| u_n\|\geq C$ for some $C>0$, then $u_0\neq 0$. 

\item Suppose that $\lambda_{\Omega}>1$. If $u_0\neq 0$, then $u_0\not\in H^1_0(\Omega)$, i.e., $u_0 \in B^{+}$. 
\end{enumerate}
\end{lem}

\begin{proof}
(i) From the inequalities
\begin{align} \label{Eu0ineq}
E(u_0)\leq \varliminf_{n}E(u_n)\leq \varlimsup_{n}E(u_n)\leq0, 
\end{align}
we infer that if $u_0=0$, then $\| u_n \|\rightarrow 0$, as desired. 

(ii) If $u_0\in H^1_0(\Omega)$, then $u_0=0$ from \eqref{Eu0ineq}, using Lemma \ref{lemD}. 
\end{proof}


The Nehari manifold for \eqref{p} is then defined by 
\begin{align*}
\mathcal{N}_{\lambda}:= \left\{ u\in H^1(\Omega)\setminus \{ 0\} : E(u)+A(u)+\lambda B(u)=0 \right\}. 
\end{align*}
It should be noted that a positive solution of \eqref{p} belongs to $\mathcal{N}_\lambda$. Given $u\neq 0$, 
the fibering map for \eqref{p} is defined as 
\begin{align*}
j_{u}(t):= J_{\lambda}(tu)
= \frac{t^2}{2}E(u)+\frac{t^{p+1}}{p+1}A(u)+\frac{\lambda t^{q+1}}{q+1}B(u), \quad t>0. 
\end{align*}
A direct computation gives us that 
\begin{align*}
j_{u}^{\prime}(t) 
= tE(u)+t^{p}A(u)+\lambda t^{q}B(u), 
\end{align*}
where the derivative of a function is represented with a prime. 
Then, we observe that 
\begin{align*}
j_{u}^{\prime}(t)=0 \ \Longleftrightarrow \ 
tu\in \mathcal{N}_{\lambda} \quad \mbox{(in particular $j_{u}^{\prime}(1)=0 \ \Longleftrightarrow \ u\in \mathcal{N}_{\lambda}$)}. 
\end{align*}

We next split $\mathcal{N}_{\lambda}$ into three parts, using $j_{u}$. 
By direct computation, 
\begin{align*}
j_{u}^{\prime\prime}(t)=E(u)+pt^{p-1}A(u)+ \lambda q t^{q-1}B(u), 
\end{align*}
and it follows that $j_{u}^{\prime\prime}(1)=E(u)+p A(u)+ \lambda q B(u)$. If $j_{u}^{\prime}(1)=0$, then we infer that 
\begin{align*}
j_{u}^{\prime\prime}(1) 
= \left\{ \begin{array}{ll}
(1-q)E(u)+(p-q)A(u),  &  \\
-(p-1)E(u)-\lambda (p-q) B(u).  & 
\end{array} \right.
\end{align*}
We then define  
\begin{align*}
\mathcal{N}_{\lambda}^{\pm}:= 
\{ u \in \mathcal{N}_\lambda : j_{u}^{\prime\prime}(1)\gtrless 0 \}, 
\end{align*}
that is, 
\begin{align*} 
\mathcal{N}_{\lambda}^{\pm}
= \left\{ u \in \mathcal{N}_\lambda : 
E(u) \gtrless -\frac{p-q}{1-q}A(u) \right\} 
= \left\{ u \in \mathcal{N}_\lambda : E(u) \lessgtr -\lambda \frac{p-q}{p-1}B(u) \right\}. 
\end{align*}

The next lemma is a direct consequence from these definitions. 
\begin{lem} \label{lem:Nlaminclu}
$\mathcal{N}_{\lambda}\subset A^+ \cap E^{-}$ and  $\mathcal{N}_{\lambda}^{-}\subset A^+ \cap B^+ \cap E^{-}$. 
Moreover, $\mathcal{N}_{\lambda}\subset A^{+}\cap B^{+} \cap E^{-}$ if $\lambda_{\Omega}>1$. 
\end{lem}

\begin{proof}
If $u\in \mathcal{N}_\lambda$, then $u\neq 0$, i.e., $u\in A^{+}$. It follows that $E(u)=-A(u)-\lambda B(u)\leq -A(u)<0$, thus, $u\in E^{-}$. Additionally if $u\in  \mathcal{N}_{\lambda}^{-}$, then it follows that $\lambda \frac{p-q}{p-1}B(u)>-E(u)>0$. Thus, $u \in B^{+}$. We assume that $u\in \mathcal{N}_{\lambda}$ and $B(u)=0$ under $\lambda_{\Omega}>1$. 
Lemma \ref{lemD} then shows that $E(u)>0$ because $u\neq 0$ and $u\in H^1_0(\Omega)$, which is contradictory for $\mathcal{N}_{\lambda}\subset E^{-}$.  Hence, we deduce that $\mathcal{N}_{\lambda}\subset B^{+}$. 
\end{proof}

Using the change of variables 
\begin{align} \label{cv}
\mu=\lambda^{\frac{p-1}{1-q}} \ \mbox{ and } \  v=\lambda^{-\frac{1}{1-q}}u, 
\end{align}
we also consider the functional 
\begin{align*} 
I_{\mu}(v):=\frac{1}{2}E(v)+\frac{\mu}{p+1}A(v)+\frac{1}{q+1}B(v), \quad v\in H^1(\Omega),      
\end{align*}
associated with the problem 
\begin{align} \label{pv}
\begin{cases}
-\Delta v = v - \mu v^p & \mbox{ in } \Omega, \\
\frac{\partial v}{\partial \nu}=-v^q & \mbox{ on } \partial\Omega.  
\end{cases}    
\end{align}
A nonnegative function $v\in H^1(\Omega)$ is called a \textit{nonnegative (weak) solution} of \eqref{p} if $v$ satisfies  
\begin{align} \label{def:vwsmu}
\int_{\Omega}\left( \nabla v \nabla \varphi - v \varphi + \mu v^p \varphi \right) + \int_{\partial\Omega} v^q \varphi =0, \quad \varphi \in H^1(\Omega). 
\end{align}
It should be noted that \eqref{pv} with $\mu=0$ is \eqref{lp}. 

The following result is the counterpart of Lemma \ref{lem:Jcoer} for $J_\lambda$. 
\begin{lem} \label{lem:Icoer}
Let $\mu>0$. Then, $I_{\mu}$ is coercive, and is bounded from below in $H^1(\Omega)$. More precisely, we have $C_{\mu}>0$ such that $I_{\mu}(v)\geq \frac{1}{2}\| v\|^2-C_{\mu}$ for $u\in H^1(\Omega)$. 
\end{lem}

Similarly, we introduce the Nehari manifold associated with \eqref{pv}: 
\begin{align*}
\mathcal{M}_{\mu}:= 
\left\{ v\in H^1(\Omega)\setminus \{ 0\} : E(v)+\mu A(v) + B(v)=0 \right\}, 
\end{align*}
and the fibering map $i_{v}(t)$ for $v\neq 0$: 
\begin{align} \label{zeta}
i_{v}(t):=I_{\mu}(tv)=\frac{t^2}{2}E(v)+\frac{\mu t^{p+1}}{p+1}A(v)+\frac{t^{q+1}}{q+1}B(v), \quad t>0.     
\end{align}
By direct computation,  
\begin{align} \label{zetaprim}
i_{v}^{\prime}(t)=tE(v) + \mu t^p A(v) + t^q B(v), 
\end{align}
and we deduce that $i_{v}^{\prime}(t)=0$ if and only if 
$tv\in \mathcal{M}_{\mu}$. Particularly, 
$i_{v}^{\prime}(1)=0 \ \Longleftrightarrow \ v\in \mathcal{M}_{\mu}$. 
Moreover, observing that 
\begin{align*}
i_{v}^{\prime\prime}(t)=E(v)+ \mu p t^{p-1} A(v) + qt^{q-1}B(v), 
\end{align*}
we define similarly 
\begin{align*}
\mathcal{M}_{\mu}^{\pm}:= \left\{ v \in \mathcal{M}_{\mu} : i_{v}^{\prime\prime}(1)\gtrless 0 \right\} 
&= \left\{ v \in \mathcal{M}_{\mu} : 
E(v) \gtrless -\mu \frac{p-q}{1-q}A(v) \right\}  \\ 
&= \left\{ v \in \mathcal{M}_{\mu} : E(v) \lessgtr 
-\frac{p-q}{p-1}B(v) \right\}. 
\end{align*}
Clearly, if $\mu=\lambda^{\frac{p-1}{1-q}}$, then 
\begin{align}
&u \in \mathcal{N}_{\lambda} \ \Longleftrightarrow \ 
v=\lambda^{-\frac{1}{1-q}}u \in \mathcal{M}_{\mu}, \label{NM1}\\
&u \in \mathcal{N}_{\lambda}^{\pm} \ \Longleftrightarrow \ 
v=\lambda^{-\frac{1}{1-q}}u \in \mathcal{M}_{\mu}^{\pm}.  \label{NM2}
\end{align}


We now look for a certain condition for $j_{u}$ to possess two distinct critical points. Given $u\in A^{+}\cap B^{+} \cap E^{-}$, we set $j_{u}^{\prime}(t)=t^q \,  \tilde{j}_{u}(t)$ with 
\begin{align*}  
\tilde{j}_{u}(t):= t^{1-q}E(u)+t^{p-q}A(u)+\lambda B(u), \quad t>0. 
\end{align*}
We note that $\tilde{j}_{u}$ has the unique global minimum point
\begin{align*} 
t_0 = t_0(u)=\left( \frac{1-q}{p-q} \right)^{\frac{1}{p-1}}\left( \frac{-E(u)}{A(u)} \right)^{\frac{1}{p-1}}>0,  
\end{align*}
and that $\tilde{j}_{u}$ is decreasing and increasing for $t<t_0$ and $t>t_0$, respectively. 
Therefore, $j_{u}$ has two distinct critical points if and only if $\tilde{j}_{u}(t_0)<0$, and in this case, $j_{u}$ possesses exactly two critical points. The desired condition is given by the following: 
\begin{align}
\lambda < \frac{p-1}{p-q}\left( \frac{1-q}{p-q}\right)^{\frac{1-q}{p-1}}
\left( \frac{-E(u)}{A(u)^{\frac{1-q}{p-q}} B(u)^{\frac{p-1}{p-q}}} \right)^{\frac{p-q}{p-1}}. \label{lam2cri}
\end{align}
On the basis of \eqref{lam2cri}, we discuss a class of $u$ for which $j_{u}$ has two distinct critical points for $\lambda>0$ small. 
We define 
\begin{align*}
\mathcal{F}_{\delta}:= \left\{ u \in A^{+}\cap B^{+} : E(u)+A(u)\leq 0, \ \| u\|\geq \delta \right\}, \quad 0<\delta\leq \frac{|\Omega|^{\frac{1}{2}}}{2}, 
\end{align*}
and introduce the value 
\begin{align} \label{lamast}
\lambda_{\ast}=\lambda_{\ast}(\delta):= \inf
\left\{ 
\frac{p-1}{p-q}\left( \frac{1-q}{p-q}\right)^{\frac{1-q}{p-1}}
\left( \frac{-E(u)}{A(u)^{\frac{1-q}{p-q}} B(u)^{\frac{p-1}{p-q}}} \right)^{\frac{p-q}{p-1}} : u\in \mathcal{F}_\delta 
\right\}. 
\end{align}
It should be noted that $u=\frac{1}{2}\in \mathcal{F}_\delta$, and we deduce the following lemma: 
\begin{lem} \label{lem:Fbdd}
$\mathcal{F}_{\delta}$ is bounded in $H^1(\Omega)$. 
\end{lem}

\begin{proof}
Assume by contradiction that $\| u_n\|\rightarrow \infty$ for $u_n \in \mathcal{F}_{\delta}$. Then, say $w_n = \frac{u_n}{\| u_n\|}$, and $\| w_n\|=1$. Thus, $E(w_n)$ is bounded. Moreover, up to a subsequence, $w_n \rightharpoonup w_0$, and $w_n \rightarrow w_0$ in $L^{p+1}(\Omega)$ and $L^2(\partial\Omega)$. Lemma \ref{lem:wlsc} (i) shows that  $w_0\neq 0$, i.e., $A(w_0)>0$. The condition $E(u_n)\leq -A(u_n)$ implies that $E(w_n)\leq -A(w_n) \| u_n\| \longrightarrow -\infty$, which is a contradiction. \end{proof}

Using Lemma \ref{lem:Fbdd}, we prove that $\lambda_{\ast}$ is positive.  
\begin{prop} \label{prop:gam0}
$\lambda_{\ast}(\delta)>0$ in any case of $\lambda_{\Omega}$. 
\end{prop}

\begin{proof} 
Assume to the contrary that $u_n \in \mathcal{F}_\delta$ admits the condition that 
\begin{align} \label{quo0}
\frac{E(u_n)}{A(u_n)^{\frac{1-q}{p-q}}B(u_n)^{\frac{p-1}{p-q}}} 
\nearrow 0. 
\end{align}
Since $u_n$ is bounded in $H^1(\Omega)$ from Lemma \ref{lem:Fbdd}, we obtain a subsequence of $\{ u_n\}$, still denoted by the same notation, such that $u_n \rightharpoonup u_0$, and $u_n \rightarrow u_0$ in $L^{p+1}(\Omega)$ and $L^2(\partial\Omega)$. Since $\| u_n \|\geq \delta$, Lemma \ref{lem:wlsc} (i) shows that $u_0\neq 0$, i.e., $A(u_0)>0$, from which it follows that $\varlimsup_{n}E(u_n)\leq -A(u_0)<0$. 
Given $\varepsilon>0$, we deduce from \eqref{quo0} that if $n$ is large enough, then 
\begin{align*} 
-\varepsilon A(u_n)^{\frac{1-q}{p-q}}B(u_n)^{\frac{p-1}{p-q}}\leq E(u_n)\leq0, 
\end{align*}
so that 
\begin{align*} 
-\varepsilon A(u_0)^{\frac{1-q}{p-q}}B(u_0)^{\frac{p-1}{p-q}}\leq \varliminf_{n}E(u_n) \leq \varlimsup_{n}E(u_n)\leq0.  
\end{align*}
This implies that $E(u_n)\rightarrow 0$ because $\varepsilon>0$ is arbitrary, which is a contradiction. 
\end{proof}

The following result is derived as a corollary from Proposition \ref{prop:gam0}. 
\begin{cor} \label{prop:twocritical}
Let $0<\delta\leq \frac{|\Omega|^{\frac{1}{2}}}{2}$, and let 
$(\lambda, u) \in (0, \lambda_{\ast}(\delta))\times \mathcal{F}_\delta$. Then, $j_{u}$ has exactly two critical points $t_1, t_2>0$, i.e., $j_{u}^{\prime}(t_j)=0$, $j=1,2$, such that $0<t_1<t_2$, and $j_{u}^{\prime\prime}(t_1)<0<j^{\prime\prime}(t_2)$. Consequently, $t_1u \in \mathcal{N}_{\lambda}^{-}$ and $t_2u\in \mathcal{N}_{\lambda}^{+}$. 
\end{cor} 

We then establish a similar result for $i_{v}$ in \eqref{zeta}. Let $v\in A^{+}\cap B^{+} \cap E^{-}$. From \eqref{cv} and \eqref{lam2cri}, we observe that $i_{v}$ has two distinct critical points if 
\begin{align*}
\mu<\frac{1-q}{p-q}\left(\frac{p-1}{p-q}\right)^{\frac{p-1}{1-q}}\left( \frac{-E(v)}{A(v)^{\frac{1-q}{p-q}}B(v)^{\frac{p-1}{p-q}}}\right)^{\frac{p-q}{1-q}},
\end{align*}
and characterize a class of $v$ for which $i_{v}$ possesses two distinct critical points for $\mu > 0$ small. We define  
\begin{align*}
\mathcal{G}_{\delta}:= \left\{ v \in A^{+}\cap B^{+} : E(v)+B(v)\leq 0, \ \| v\|\leq \delta \right\},  \quad 
\delta \geq \frac{|\partial\Omega|^\frac{1}{1-q}}{|\Omega|^{\frac{1+q}{2(1-q)}}},  
\end{align*}
and introduce the value 
\begin{align} \label{muastdef}
\mu_{\ast}=\mu_{\ast}(\delta):= \inf\left\{ 
\frac{1-q}{p-q}\left(\frac{p-1}{p-q}\right)^{\frac{p-1}{1-q}}\left( \frac{-E(v)}{A(v)^{\frac{1-q}{p-q}}B(v)^{\frac{p-1}{p-q}}}\right)^{\frac{p-q}{1-q}}: v\in \mathcal{G}_\delta \right\}. 
\end{align}
It should be noted that $v=\left( \frac{|\partial \Omega|}{|\Omega|}\right)^{\frac{1}{1-q}}\in \mathcal{G}_{\delta}$, and we obtain the following lemma: 


\begin{lem} \label{lem:Gdelbddbe}
Assume that $\lambda_{\Omega}>1$. Then, there exists $C>0$ such that 
$\| v \|\geq C$ for $v \in\mathcal{G}_{\delta}$. 
\end{lem}

\begin{proof}
Assume by contradiction that $\| v_n\|\rightarrow 0$ for $v_n \in \mathcal{G}_{\delta}$. 
Say $w_n=\frac{v_n}{\| v_n\|}$, and $\| w_n\|=1$. Then, $E(w_n)$ is bounded, and moreover, up to a subsequence, $w_n \rightharpoonup w_0$, and $w_n \rightarrow w_0$ in $L^{p+1}(\Omega)$ and $L^2(\partial\Omega)$. Lemma \ref{lem:wlsc} shows that $w_0\neq 0$ and $w_0\not\in H^1_0(\Omega)$, i.e., $B(w_0)>0$. 
The condition $E(v_n)\leq -B(v_n)$ implies that $E(w_n)\leq - B(w_n)\| v_n \|^{-(1-q)} \rightarrow -\infty$, as desired. 
\end{proof}

The next proposition is the counterpart of Proposition \ref{prop:gam0} for $\lambda_{\ast}(\delta)$.  
\begin{prop} \label{prop:muast}
$\mu_{\ast}(\delta)>0$ if $\lambda_{\Omega}>1$.  
\end{prop}

\begin{proof}
The proof is similar as that of Proposition \ref{prop:gam0}. 
Assume by contradiction that $v_n \in \mathcal{G}_{\delta}$, and 
\begin{align} \label{EABv0}
\frac{E(v_n)}{A(v_n)^{\frac{1-q}{p-q}}B(v_n)^{\frac{p-1}{p-q}}} \nearrow 0.  
\end{align}
Since $v_n$ is bounded in $H^1(\Omega)$, it follows that up to a subsequence, $v_n \rightharpoonup v_0$, and $v_n \rightarrow v_0$ in $L^{p+1}(\Omega)$ and $L^2(\partial\Omega)$. 
Additionally, because $\| v_n\|\geq C$ from Lemma \ref{lem:Gdelbddbe}, Lemma \ref{lem:wlsc} ensures that $v_0 \neq 0$ and $v_0\not\in H^1_0(\Omega)$, i.e., $B(v_0)>0$, which implies that 
\begin{align} \label{1526}
\varlimsup_{n}E(v_n)\leq -B(v_0)<0. 
\end{align}
Given $\varepsilon>0$, we infer from \eqref{EABv0} that if $n$ is large enough, then 
\begin{align*}
-\varepsilon A(v_n)^{\frac{1-q}{p-q}}B(v_n)^{\frac{p-1}{p-q}}\leq E(v_n)\leq 0,  
\end{align*}
thus, 
\begin{align*} 
-\varepsilon A(v_0)^{\frac{1-q}{p-q}}B(v_0)^{\frac{p-1}{p-q}}\leq \varliminf_{n} E(v_n) \leq \varlimsup_{n} E(v_n)\leq 0,  
\end{align*}
which implies that $E(v_n)\rightarrow 0$ because $\varepsilon>0$ is arbitrary. This is contradictory for \eqref{1526}.  \end{proof}

The next result is then derived as a corollary from Proposition \ref{prop:muast}, which is the counterpart of Corollary \ref{prop:twocritical} for $j_{u}$. 
\begin{cor} \label{cor:twocri-v}
Let $\delta\geq  |\partial\Omega|^{\frac{1}{1-q}}/|\Omega|^{\frac{1+q}{2(1-q)}}$, and let $(\mu, v)\in (0, \mu_{\ast}(\delta))\times  \mathcal{G}_\delta$. Then, $i_{v}$ has exactly two critical points $t_1, t_2>0$, i.e., $i_{v}^{\prime}(t_j)=0$, $j=1,2$, such that $0<t_1<t_2$, and $i_{v}^{\prime\prime}(t_1)<0<i_{v}^{\prime\prime}(t_2)$. Consequently, $t_1v \in \mathcal{M}_{\mu}^{-}$ and $t_2v\in \mathcal{M}_{\mu}^{+}$. 
\end{cor} 


\subsection{Existence of a global minimizer on $\mathcal{N}_{\lambda}^{+}$} 

First, we claim that $\mathcal{N}_{\lambda}^{+}\neq \emptyset$ for $\lambda > 0$ small. 
 
\begin{lem} \label{lem:Nlam+non}
There exists $\lambda_0>0$ such that if $\lambda\in (0,\lambda_0)$, then we have a 
unique constant $c_+(\lambda) \in \mathcal{N}_{\lambda}^{+}$ such that $c_+(\lambda)<1$ and $c_+(\lambda)\nearrow 1$ as $\lambda \to 0^{+}$. Moreover, it holds that 
\begin{align} \label{Jlambdd}
\sup_{\lambda\in (0,\lambda_0)} J_{\lambda}(c_+(\lambda))\leq  -\frac{p-1}{3(p+1)}|\Omega|. 
\end{align}
\end{lem}

\begin{proof}
For a constant $c>0$, it is easy to observe that $c\in \mathcal{N}_{\lambda}^{+}$ if and only if 
\begin{align*}
& c^{1-q}-c^{p-q} = \lambda \frac{|\partial \Omega|}{|\Omega|}, \quad \mbox{and} \ \ 1<\frac{p-q}{1-q} c^{p-1}. 
\end{align*}
Hence, the first assertion holds for $\lambda>0$ sufficiently small. For \eqref{Jlambdd}, the following calculation is conducted:  
\begin{align*}
J_{\lambda}(c_+(\lambda)) 
&= \left( \frac{1}{2} - \frac{1}{p+1}\right)E(c_+(\lambda)) + \lambda \left( \frac{1}{q+1}-\frac{1}{p+1}\right)B(c_+(\lambda)) \\ 
&=c_+(\lambda)^2 \left( -\frac{p-1}{2(p+1)}|\Omega| + \lambda \frac{p-q}{(p+1)(q+1)}|\partial\Omega| c_+(\lambda)^{q-1} \right) \\ 
& \longrightarrow -\frac{p-1}{2(p+1)}|\Omega| \quad \mbox{ as } \lambda \to 0^{+}. 
\end{align*}
The desired conclusion now follows. \end{proof}

For $\lambda \in (0, \lambda_0)$, we deduce from Lemma \ref{lem:Jcoer} that 
\begin{align} \label{gamdef}
\eta_{\lambda}^{+}:= \inf\left\{ J_{\lambda}(u) : u\in  \mathcal{N}_{\lambda}^{+} \right\}>-\infty.     
\end{align}
In this subsection, we establish the following result.
\begin{prop} \label{prop:min+} 
Assume that $\lambda_{\Omega}>1$. Then, there exists $\overline{\lambda}_{+}\in (0, \lambda_0)$ such that for $\lambda \in (0, \overline{\lambda}_{+})$, 
\begin{align*}
\eta_{\lambda}^{+} = J_{\lambda}(u_{\lambda}^{+}) = \min\left\{ J_{\lambda}(u) : u\in \mathcal{N}_{\lambda}^{+} \right\}<0,      
\end{align*}
and additionally, there exists $C>1$ such that 
\begin{align} \label{ulaminftbdd}
C^{-1}\leq \| u_{\lambda}^{+}\|\leq C \quad \mbox{as} \ \  \lambda \to 0^{+}.
\end{align}
\end{prop}

From \eqref{gamdef}, let $\{ u_{\lambda, n} \}\subset \mathcal{N}_{\lambda}^{+}$ be a minimizing sequence for $J_{\lambda}$ on $\mathcal{N}_{\lambda}^{+}$ such that $J_{\lambda}(u_{\lambda, n}) \searrow \eta_{\lambda}^{+}$. Lemma \ref{lem:Jcoer} then ensures that up to a subsequence, 
\begin{align} \label{miniseq+}
u_{\lambda, n}\rightharpoonup u_{\lambda,\infty}, \ \mbox{ and } \  
u_{\lambda,n} \rightarrow u_{\lambda,\infty} \ \mbox{ in } L^{p+1}(\Omega) \ 
\mbox{ and in } \ L^{2}(\partial\Omega).
\end{align}
It follows from \eqref{Jlambdd} that $\| u_{\lambda,\infty} \|$ has an {\it a priori} lower bound if $\lambda > 0$ is small enough:
\begin{lem} \label{lem:lbulam}
Let $u_{\lambda,\infty}$ be as in \eqref{miniseq+}. 
Then, there exist $\delta_+>0$ and $\lambda_+>0$ such that $\| u_{\lambda,\infty} \|\geq \delta_+$ for $\lambda\in (0,\lambda_+)$, where $\lambda_+, \delta_+>0$ do not depend on the choice of $u_{\lambda,\infty}$. 
\end{lem}

\begin{proof}
It follows from \eqref{miniseq+} that  $J_{\lambda}(u_{\lambda,\infty})\leq \varliminf_{n}J_{\lambda}(u_{\lambda, n})=\eta_{\lambda}^{+}$. Therefore, assertions \eqref{Jlambdd} and \eqref{gamdef} provide that 
\begin{align} \label{Jlamulam-}
J_{\lambda}(u_{\lambda,\infty})\leq \eta_{\lambda}^{+} 
\leq J_{\lambda}(c_+(\lambda))\leq -\frac{p-1}{3(p+1)}|\Omega| \quad \mbox{ for $\lambda \in (0,\lambda_0)$.}
\end{align}
The desired conclusion follows. 
\end{proof}

We then prove Proposition \ref{prop:min+}. 
\begin{proof}[Proof of Proposition \ref{prop:min+}]
With $\lambda_+, \delta_+>0$ of Lemma \ref{lem:lbulam} and $\lambda_{\ast}(\delta_+)$ by \eqref{lamast}, we fix  
\[
0< \lambda < \min\left( 
\lambda_+, \, \lambda_{\ast}(\delta_+)
\right). 
\]
Let $u_{\lambda,n}$ and $u_{\lambda,\infty}$ be as in \eqref{miniseq+}. First, we verify that $u_{\lambda,\infty}\in \mathcal{F}_{\delta_{+}}$, and apply Corollary \ref{prop:twocritical} with $\delta=\delta_{+}$. 
We may infer that $\delta_+ \leq \frac{|\Omega|^{\frac{1}{2}}}{2}$. From Lemma \ref{lem:lbulam}, we note that $\| u_{\lambda,\infty} \|\geq \delta_+$, thus, $u_{\lambda,\infty}\neq 0$, i.e., $u_{\lambda,\infty}\in A^+$. Using $u_{\lambda,n}\in \mathcal{N}_{\lambda}$, we deduce that 
\begin{align*}
E(u_{\lambda,\infty})\leq \varliminf_{n}E(u_{\lambda, n})= \varliminf_{n}(-A(u_{\lambda, n})-\lambda B(u_{\lambda,n})) \leq  -A(u_{\lambda,\infty}), 
\end{align*}
thus, $E(u_{\lambda,\infty})+A(u_{\lambda,\infty})\leq0$. Moreover, we apply Lemma \ref{lem:wlsc} (ii) to obtain that  $u_{\lambda,\infty}\not\in H^1_0(\Omega)$, i.e., $u_{\lambda,\infty}\in B^{+}$, considering \eqref{miniseq+} and the condition that $\lambda_{\Omega}>1$, as desired. 
Corollary \ref{prop:twocritical} with $\delta=\delta_{+}$ now applies, and then, there exist $0<t_1<t_2$ such that $t_1u_{\lambda,\infty} \in \mathcal{N}_{\lambda}^{-}$ 
and $t_2 u_{\lambda,\infty} \in \mathcal{N}_{\lambda}^{+}$. 

Next, we prove that $u_{\lambda, n} \rightarrow u_{\lambda,\infty}$ in $H^1(\Omega)$. 
If not, using $u_{\lambda,\infty}\in A^{+}\cap B^{+}\cap E^{-}$, we then infer that $t_1<1<t_2$ because we have a subsequence of $\{ u_{\lambda,n}\}$, still denoted by the same notation, such that 
\begin{align*}
j_{u_{\lambda,\infty}}^{\prime}(1)=E(u_{\lambda,\infty})+A(u_{\lambda,\infty})+\lambda B(u_{\lambda,\infty}) 
< \lim_{n} (E(u_{\lambda, n}) + A(u_{\lambda, n}) + \lambda B(u_{\lambda, n})) = 0. 
\end{align*}
Hence, we deduce that 
\begin{align*}
J_{\lambda}(t_2 u_{\lambda,\infty}) = j_{u_{\lambda,\infty}}(t_2) < j_{u_{\lambda,\infty}}(1) \leq  \varliminf_{n}j_{u_{\lambda, n}}(1) = \varliminf_{n}J_{\lambda}(u_{\lambda, n})=\eta_{\lambda}^{+},  
\end{align*}
which is contradictory for $t_2 u_{\lambda,\infty}\in\mathcal{N}_{\lambda}^{+}$, as desired. Immediately, it follows that $J_{\lambda}(u_{\lambda, n}) \rightarrow J_{\lambda}(u_{\lambda,\infty})=\eta_{\lambda}^{+}$.

Finally, we verify that $t_2=1$. To this end, we only have to notice that if $0=j_{u_{\lambda,\infty}}^{\prime}(t)\leq j_{u_{\lambda,\infty}}^{\prime\prime}(t)$, then $t>0$ is unique, 
and $j_{u_{\lambda,\infty}}^{\prime\prime}(t)>0$. We infer that $0=j_{u_{\lambda, n}}^{\prime}(1)<j_{u_{\lambda, n}}^{\prime\prime}(1)$. 
Passing to the limit provides that $0=j_{u_{\lambda,\infty}}^{\prime}(1)\leq j_{u_{\lambda,\infty}}^{\prime\prime}(1)$, thus, $j_{u_{\lambda,\infty}}^{\prime\prime}(1)>0$, which means that $t_2=1$, i.e., $u_{\lambda,\infty}\in \mathcal{N}_{\lambda}^{+}$. By \eqref{Jlamulam-}, we observe that $\eta_{\lambda}^{+}=J_{\lambda}(u_{\lambda,\infty})<0$, which completes the proof with $u_{\lambda}^{+}=u_{\lambda,\infty}$. Indeed, \eqref{ulaminftbdd} follows from the fact that $u_{\lambda,\infty}\in \mathcal{F}_{\delta_{+}}$ combined with Lemma \ref{lem:Fbdd}.  \end{proof}


\subsection{Existence of a global minimizer on $\mathcal{N}_{\lambda}^{-}$} 
First, we prove that $\mathcal{M}_{\mu}^{-}\neq \emptyset$ for $\mu>0$ small, 
which shows that $\mathcal{N}_{\lambda}^{-}\neq \emptyset$ for $\lambda>0$ small. 
\begin{lem} \label{lem:Mmu-non}
There exists $\mu_0>0$ such that if $\mu\in (0,\mu_0)$, then we have a 
unique constant $c_{-}(\mu)\in \mathcal{M}_{\mu}^{-}$ such that $c_{-}(\mu)> \left( \frac{|\partial\Omega|}{|\Omega|}\right)^{\frac{1}{1-q}}$, and $c_{-}(\mu) \searrow \left( \frac{|\partial\Omega|}{|\Omega|}\right)^{\frac{1}{1-q}}$ as $\mu \to 0^{+}$. Moreover, it holds that 
\begin{align} \label{Imubdd}
\sup_{\mu\in (0,\mu_0)}I_{\mu}(c_{-}(\mu)) < 
\left( \frac{1-q}{1+q} \right) 
\frac{|\partial\Omega|^{\frac{2}{1-q}}}{|\Omega|^{\frac{1+q}{1-q}}}. 
\end{align}
\end{lem}

\begin{proof}
Consider a positive constant $c\in \mathcal{M}_{\mu}^{-}$, i.e., $E(c)+\mu A(c) + B(c)=0$ and $E(c)>-\frac{p-q}{p-1}B(c)$. Observing that 
\begin{align*}
c\in \mathcal{M}_{\mu}^{-} \ \Longleftrightarrow \ 
\left\{ 
\begin{array}{l}
\mu=c^{-(p-1)}-c^{-(p-q)}\frac{|\partial\Omega|}{|\Omega|}, \medskip \\
c< \left( \frac{p-q}{p-1} \right)^{\frac{1}{1-q}}\left( \frac{|\partial\Omega|}{|\Omega|} \right)^{\frac{1}{1-q}},     
\end{array} \right.
\end{align*}
the first assertion holds for $\mu>0$ sufficiently small. For \eqref{Imubdd}, we conduct the calculation 
\begin{align*}
I_{\mu}(c_{-}(\mu))
&=\left( \frac{1}{2} - \frac{1}{q+1} \right)E(c_{-}(\mu)) + \mu \left( \frac{1}{p+1} - \frac{1}{q+1}\right) A(c_{-}(\mu)) \\ 
&=c_{-}(\mu)^2|\Omega| \left\{ \frac{1-q}{2(q+1)} -\mu \frac{p-q}{(p+1)(q+1)}c_{-}(\mu)^{p-1}\right\}    \\ 
& \leq c_{-}(\mu)^2|\Omega| \frac{1-q}{2(q+1)} \longrightarrow 
\frac{|\partial\Omega|^{\frac{2}{1-q}}}{|\Omega|^{\frac{1+q}{1-q}}}\frac{1-q}{2(q+1)} \quad \mbox{ as } \mu \to 0^{+}.  
\end{align*}
The desired conclusion now follows. \end{proof}

In this subsection, we establish the following result:
\begin{prop} \label{prop:minJlam-}
Assume that $\lambda_{\Omega}>1$. Then, there exists $\overline{\lambda}_{-}>0$ such that for $\lambda \in (0, \overline{\lambda}_{-})$, 
\begin{align*}
J_{\lambda}(u_{\lambda}^{-}) = \min\left\{ J_{\lambda}(u) : u\in \mathcal{N}_{\lambda}^{-} \right\}>0,      
\end{align*}
and additionally, there exists $C>1$ such that
\begin{align} \label{ulamminu:asympt}
C^{-1}\lambda^{\frac{1}{1-q}} \leq \| u_{\lambda}^{-}\| \leq C \lambda^{\frac{1}{1-q}} \quad\mbox{ as } \ \lambda \to 0^{+}. 
\end{align}
\end{prop}
For this purpose, we discuss the existence of a minimizer for the functional $I_{\mu}$ on $\mathcal{M}_{\mu}^{-}$. First, we verify that $I_{\mu}$ is nonnegative on $\mathcal{M}_{\mu}^{-}$. 
\begin{lem} \label{lem:JB}
Let $\mu > 0$. If $v \in \mathcal{M}_{\mu}^{-}$, then 
$I_{\mu}(v)\geq \frac{(p-q)(1-q)}{2(p+1)(q+1)}B(v)$. 
\end{lem}

\begin{proof}
Using the fact that $\mu A(v)=-E(u)-B(u)$ and $E(v)>-\frac{p-q}{p-1} B(v)$ for $v\in \mathcal{M}_{\mu}^{-}$, we deduce the assertion by direct calculations. \end{proof} 

Let $\mu\in (0,\mu_0)$. From Lemma \ref{lem:JB}, we define 
\begin{align} \label{ximu-}
\xi_{\mu}^{-}:= \inf\left\{   I_{\mu}(v) : v\in \mathcal{M}_{\mu}^{-} \right\}
\geq 0.    
\end{align}
Let $\{v_{\mu,n}\}\subset \mathcal{M}_{\mu}^{-}$ be a minimizing sequence for $I_{\mu}$ on $\mathcal{M}_{\mu}^{-}$ such that 
$I_{\mu}(v_{\mu,n})\searrow \xi_{\mu}^{-}$. Then, it follows from Lemma \ref{lem:Icoer} that 
up to a subsequence, 
\begin{align} \label{vmu}
v_{\mu,n}\rightharpoonup v_{\mu,\infty}, \ \ \mbox{ and } \ \ 
v_{\mu,n} \rightarrow v_{\mu,\infty} \ \mbox{ in $L^{p+1}(\Omega)$ and in  $L^{2}(\partial\Omega)$}. 
\end{align}

We then prove the following result:
\begin{prop} \label{prop:minImu}
Assume that $\lambda_{\Omega}>1$. Then, there exists $\overline{\mu}_{-}>0$ such that if $\mu \in (0,  \overline{\mu}_{-})$, then 
\begin{align*} 
\xi_{\mu}^{-}=
I_{\mu}(v_{\mu}^{-})= \min\left\{ I_{\mu}(v) : v\in \mathcal{M}_{\mu}^{-} \right\}>0, 
\end{align*}
and additionally, there exists $C>1$ such that 
\begin{align} \label{vmu-bdd}
C^{-1}\leq \| v_{\mu}^{-}\| \leq C \quad\mbox{ as } \ \mu \to 0^{+}. 
\end{align}
\end{prop}

As a matter of fact, in view of \eqref{NM1} and \eqref{NM2}, Proposition \ref{prop:minJlam-} is a direct consequence of Proposition \ref{prop:minImu}. Indeed, $u_{\lambda}^{-}=\lambda^{\frac{1}{1-q}}v_{\mu}^{-}$ with $\mu=\lambda^{\frac{p-1}{1-q}}$.

First, we prove the following lemma: 
\begin{lem} \label{lem:vmunbdd}
Assume that $\lambda_{\Omega}>1$. Let $\mu\in (0,\mu_0)$, and let 
$\xi_{\mu}^{-}$ be as in \eqref{ximu-}. Then, $\xi_{\mu}^{-}>0$. 
\end{lem}

\begin{proof}
Let $v_{\mu,n}, v_{\mu,\infty}$ be as in \eqref{vmu}. Lemma \ref{lem:JB} shows that 
\begin{align*}
\frac{(p-q)(1-q)}{2(p+1)(q+1)}B(v_{\mu,\infty}) 
= \lim_{n} \frac{(p-q)(1-q)}{2(p+1)(q+1)}B(v_{\mu,n}) \leq 
\lim_{n} I_{\mu}(v_{\mu,n})=\xi_{\mu}^{-}. 
\end{align*}
Therefore, the desired assertion holds if $v_{\mu,\infty}\not\in H^1_0(\Omega)$. 
Since $v_{\mu,n}\in \mathcal{M}_{\mu}$, we infer that 
\begin{align*}
E(v_{\mu,\infty})\leq \varliminf_{n}E(v_{\mu,n})\leq \varlimsup_{n}E(v_{\mu,n})\leq 0. 
\end{align*}
If $v_{\mu,\infty}=0$, then $v_{\mu,n}\rightarrow 0$ in $H^1(\Omega)$. Say $w_{n}=\frac{v_{\mu,n}}{\| v_{\mu,n}\|}$, and $\| w_n\|=1$. Then, up to a subsequence, $w_{n} \rightharpoonup \hat{w}_{\infty}$, and 
$w_{n}\rightarrow \hat{w}_{\infty}$ in $L^{p+1}(\Omega)$ and $L^2(\partial\Omega)$. Lemma \ref{lem:wlsc} shows that $\hat{w}_{\infty}\neq 0$ and $\hat{w}_{\infty}\not\in H^1_0(\Omega)$, i.e., $B(\hat{w}_{\infty})>0$.  From the condition that $E(v_{\mu,n})<-B(v_{\mu,n})$, we deduce that $E(w_{n})<-B(w_{n})\| v_{\mu,n}\|^{q-1} \rightarrow -\infty$, which is a contradiction because $E(w_{n})$ is bounded. 
We thus obtain that $v_{\mu,\infty}\neq 0$. Lemma \ref{lem:wlsc} (ii) is used again to obtain that $v_{\mu,\infty}\not\in H^1_0(\Omega)$. The proof is complete. \end{proof}

Next, we prove that $v_{\mu,\infty}$ has an {\it a priori} upper bound if $\mu>0$ is small, which is the counterpart of Lemma \ref{lem:lbulam} for $u_{\lambda,\infty}$.  
\begin{lem} \label{lem:vmuub}
Assume that $\lambda_{\Omega}>1$. Let $v_{\mu,\infty}$ be as in \eqref{vmu}. Then, there exist $\delta_{-}>0$ and $\mu_{-}\in (0,\mu_0)$ such that $\| v_{\mu,\infty}\|\leq \delta_{-}$ for $\mu\in (0, \mu_{-})$, where $\mu_{-}, \delta_{-}>0$ do not depend on the choice of $v_{\mu,\infty}$.  
\end{lem}

\begin{proof}
We argue by contradiction. Assume that $\| v_{\mu_n,\infty} \| \rightarrow \infty$ for $\mu_n \to 0^{+}$. Let $n\geq1$ be fixed. Following \eqref{vmu}, we choose a sequence $\{ v_{\mu_n, k}\}_k \subset \mathcal{M}_{\mu_n}^{-}$ such that 
$v_{\mu_n,k} \rightharpoonup v_{\mu_n,\infty}$, and $v_{\mu_n,k} \rightarrow v_{\mu_n,\infty}$ in $L^{p+1}(\Omega)$ and $L^2(\partial\Omega)$. 
Then, it follows from \eqref{Imubdd} and \eqref{ximu-} that 
\begin{align*} 
I_{\mu_n}(v_{\mu_n,\infty})\leq \varliminf_{k}I_{\mu_n}(v_{\mu_n,k})
=\xi_{\mu_n}^{-} \leq I_{\mu_n}(c_{-}(\mu_n)) 
\leq -\varepsilon_0 + \left( \frac{1-q}{1+q} \right) 
\frac{|\partial\Omega|^{\frac{2}{1-q}}}{|\Omega|^{\frac{1+q}{1-q}}}. 
\end{align*}
for some $\varepsilon_0>0$. This implies that up to a subsequence of $\{ v_{\mu_n,k}\}_{k}$,  
\begin{align*}
I_{\mu_n}(v_{\mu_n,k})< \left( \frac{1-q}{1+q} \right) 
\frac{|\partial\Omega|^{\frac{2}{1-q}}}{|\Omega|^{\frac{1+q}{1-q}}}.
\end{align*}
Using Lemma \ref{lem:JB}, it follows that 
\begin{align*}
\frac{(p-q)(1-q)}{2(p+1)(q+1)}B(v_{\mu_n,k})<\left( \frac{1-q}{1+q} \right) 
\frac{|\partial\Omega|^{\frac{2}{1-q}}}{|\Omega|^{\frac{1+q}{1-q}}}. 
\end{align*}
Passing to the limit $k\to \infty$ shows that 
\begin{align} \label{Bvmun}
\frac{(p-q)(1-q)}{2(p+1)(q+1)}B(v_{\mu_n,\infty})\leq \left( \frac{1-q}{1+q} \right) 
\frac{|\partial\Omega|^{\frac{2}{1-q}}}{|\Omega|^{\frac{1+q}{1-q}}}.  
\end{align} 
Say $w_{n}=\frac{v_{\mu_n,\infty}}{\| v_{\mu_n,\infty}\|}$, and $\| w_n\|=1$. We then have a subsequence of $\{ w_n\}$, still denoted by the same notation, such that $w_{n} \rightharpoonup w_{\infty}$, and $w_{n} \rightarrow w_{\infty}$ in $L^{p+1}(\Omega)$ and $L^2(\partial\Omega)$. Lemma \ref{lem:wlsc} ensures that $w_{\infty}\neq 0$ and $w_\infty \not\in H^1_0(\Omega)$, i.e.,  $B(w_{\infty})>0$. Observing from \eqref{Bvmun} that 
\begin{align*}
\frac{(p-q)(1-q)}{2(p+1)(q+1)}B(w_{n})\leq \left( \frac{1-q}{1+q} \right) 
\frac{|\partial\Omega|^{\frac{2}{1-q}}}{|\Omega|^{\frac{1+q}{1-q}}} 
\| v_{\mu_{n},\infty}\|^{-(q+1)} \longrightarrow 0, 
\end{align*}
we deduce that $B(w_{\infty})=0$, as desired. \end{proof}

We are then ready to prove Proposition \ref{prop:minImu}. 
\begin{proof}[Proof of Proposition \ref{prop:minImu}]
With $\mu_{-}, \delta_{-}$ of Lemma \ref{lem:vmuub} and 
$\mu_{\ast}(\delta_{-})$ by \eqref{muastdef}, we fix 
\begin{align*}
0< \mu < \min\left( \mu_{-}, \, \mu_{\ast}(\delta_{-}) \right). 
\end{align*}
Let $v_{\mu,n}, v_{\mu,\infty}$ be as in \eqref{vmu}, and we verify that $v_{\mu,\infty}\in \mathcal{G}_{\delta_{-}}$ for applying Corollary \ref{cor:twocri-v} with $\delta=\delta_{-}$. We may infer that $\delta_{-}\geq  |\partial\Omega|^{\frac{1}{1-q}}/|\Omega|^{\frac{1+q}{2(1-q)}}$, and Lemma \ref{lem:vmuub} shows that $\| v_{\mu,\infty} \|\leq \delta_{-}$. We prove that $v_{\mu,\infty}\in A^{+}\cap B^{+}$ and $E(v_{\mu,\infty})+B(v_{\mu,\infty})\leq0$. 
Since $v_{\mu,n} \in E^{-}$, it follows that $E(v_{\mu,\infty})\leq \varliminf_{n}E(v_{\mu,n})\leq \varlimsup_{n}E(v_{\mu,n})\leq 0$. If $v_{\mu,\infty}=0$, then $\| v_{\mu,n} \|\rightarrow 0$. 
Say $w_{n} = \frac{v_{\mu,n}}{\| v_{\mu,n} \|}$, and $\| w_n \|=1$. We then obtain a subsequence of $\{ w_n\}$, still denoted by the same notation, such that $w_{n} \rightharpoonup w_\infty$, and $w_n \rightarrow w_\infty$ in $L^{p+1}(\Omega)$ 
and $L^2(\partial\Omega)$. 
Immediately, Lemma \ref{lem:wlsc} shows that $w_{\infty}\neq 0$ and $w_\infty \not\in H^1_0(\Omega)$. From $v_{\mu,n}\in \mathcal{M}_{\mu}$, we observe that  
\begin{align*}
I_{\mu}(v_{\mu,n}) = \left( \frac{1}{2} - \frac{1}{q+1} \right)E(v_{\mu,n}) + \mu\left(\frac{1}{p+1} - \frac{1}{q+1} \right)A(v_{\mu,n}),     
\end{align*}
and Lemma \ref{lem:JB} deduces that 
\begin{align*}
\frac{(p-q)(1-q)}{2(p+1)(q+1)}B(v_{\mu,n})\leq \left( \frac{1}{2} - \frac{1}{q+1} \right)E(v_{\mu,n})+\left(\frac{1}{p+1} - \frac{1}{q+1} \right)A(v_{\mu,n}). 
\end{align*}
It follows that 
\begin{align*}
\frac{(p-q)(1-q)}{2(p+1)(q+1)}B(w_n)
\leq \left( \frac{1}{2} - \frac{1}{q+1} \right)E(w_n)\| v_{\mu,n} \|^{1-q} 
+ \left(\frac{1}{p+1} - \frac{1}{q+1} \right)A(w_n)\| v_{\mu,n} \|^{p-q}.      
\end{align*}
Passing to the limit provides that $B(w_{\infty})=0$, which is contradictory for 
$w_\infty \not\in H^1_0(\Omega)$. 
Consequently, $v_{\mu,\infty}\neq 0$, i.e., $v_{\mu,\infty}\in A^{+}$. 
Moreover, Lemma \ref{lem:wlsc} (ii) is used again to deduce that 
$v_{\mu,\infty}\not\in H^1_0(\Omega)$, i.e., $v_{\mu,\infty}\in B^{+}$. Finally, observing that 
\begin{align*} 
E(v_{\mu,\infty})+\mu A(v_{\mu,\infty})+ B(v_{\mu,\infty})\leq 
\varliminf_{n}\left( E(v_{\mu,n})+\mu A(v_{\mu,n})+ B(v_{\mu,n}) \right)=0, 
\end{align*}
we obtain that $E(v_{\mu,\infty})+B(v_{\mu,\infty})\leq0$, as desired. Corollary \ref{cor:twocri-v} now applies, and then, there exist $0<t_1<t_2$ such that $t_1 v_{\mu,\infty}\in \mathcal{M}_{\mu}^{-}$ and $t_2 v_{\mu,\infty}\in \mathcal{M}_{\mu}^{+}$.

We then verify that 
\begin{align} \label{str4.6}
v_{\mu,n} \longrightarrow v_{\mu,\infty} \quad\mbox{ in } \ H^1(\Omega).
\end{align}
It should be noted that the inclusion $\mathcal{M}_{\mu}^{-}\subset A^{+}\cap B^{+} \cap E^{-}$ holds similarly as in Lemma \ref{lem:Nlaminclu}. 
If \eqref{str4.6} does not hold, because $v_{\mu,n}\in \mathcal{M}_{\mu}^{-}$, we then  infer that up to a subsequence, 
\begin{align} \label{t1leq1}
I_{\mu}(t_1 v_{\mu,\infty})= i_{v_{\mu,\infty}}(t_1) 
<\lim_{n}i_{v_{\mu,n}}(t_1)
\leq \lim_{n} i_{v_{\mu,n}}(1) 
=\lim_{n}I_{\mu}(v_{\mu,n})=\xi_{\mu}^{-}, 
\end{align}
provided that $t_1\leq 1$. In fact, \eqref{t1leq1} holds even if $t_1>1$. From \eqref{zetaprim}, we set $i_{v_{\mu,n}}^{\prime}(t)=t^q \, \tilde{i}_{n}(t)$ with 
\begin{align*}
\tilde{i}_{n}(t)=t^{1-q}E(v_{\mu,n})+\mu t^{p-q}A(v_{\mu,n})+ B(v_{\mu,n}). 
\end{align*}
If $\tilde{i}_{n}(t_1)<0$ for $n$ large enough, then $i_{v_{\mu,n}}(t_1)<i_{v_{\mu,n}} (1)$, thus, \eqref{t1leq1} proceeds well. To this end, we use the condition that $v_{\mu,n}\in \mathcal{M}_{\mu}$ to deduce that 
\begin{align*}
\tilde{i}_{n}(t_1)&=\mu (t_1^{p-q}-t_1^{1-q})A(v_{\mu,n})+(1-t_1^{1-q})B(v_{\mu,n}) \\
&\longrightarrow \mu (t_1^{p-q}-t_1^{1-q})A(v_{\mu,\infty})+(1-t_1^{1-q}) B(v_{\mu,\infty})=:\tilde{i}_{\infty}(t_1).   
\end{align*}
If $\tilde{i}_{\infty}(t_1)<0$, then this is the case that we desire. We use the condition that $t_1 v_{\mu,\infty}\in \mathcal{M}_{\mu}^{-}$, i.e., $\mu A(v_{\mu,\infty})< \frac{1-q}{p-1} \, t_1^{q-p}B(v_{\mu,\infty})$ to deduce that 
\begin{align*}
\tilde{i}_{\infty}(t_1)< h(t_1) B(v_{\mu,\infty}), \quad \mbox{ where }  \  
h(t):= \frac{p-q}{p-1} -\frac{1-q}{p-1}t^{-(p-1)}-t^{1-q}, \quad t>1. 
\end{align*}
Observing that $h(1)=0$ and 
\begin{align*}
h^{\prime}(t) = (1-q)(t^{-p} - t^{-q}) < 0, \quad t>1, 
\end{align*}
we find that $\tilde{i}_{\infty}(t_1)<0$, as desired, and thus, \eqref{t1leq1} proceeds for any case of $t_1>0$. However, this is contradictory for $t_1 v_{\mu,\infty}\in  \mathcal{M}_{\mu}^{-}$. Claim \eqref{str4.6} is thus verified. Immediately, we deduce from Lemma \ref{lem:vmunbdd} that $I_{\mu}(v_{\mu,n})\rightarrow I_{\mu}(v_{\mu,\infty})=\xi_{\mu}^{-}>0$.

The assertion that $t_1=1$ can be verified similarly as in the proof of Proposition \ref{prop:min+}. Thus, $v_{\mu,\infty}\in \mathcal{M}_{\mu}^{-}$, which completes the proof with $v_{\mu}^{-}=v_{\mu,\infty}$, since \eqref{vmu-bdd} follows from the fact that $v_{\mu,\infty}\in \mathcal{G}_{\delta_{-}}$ combined with Lemma \ref{lem:Gdelbddbe}. \end{proof}


\section{Asymptotic profiles of positive solutions}

\label{sec:asympt}

\subsection{Asymptotic profiles of positive solutions as $\lambda\to\infty$} 

Proposition \ref{prop:subsuper} proves that when $\lambda_{\Omega}<1$, problem \eqref{p} has a positive solution $u>0$ in $\overline{\Omega}$ for every $\lambda>0$. In this subsection, we evaluate the asymptotic profile of the positive solution as $\lambda\to\infty$.
\begin{prop} \label{prop:asympinfty}
Assume that $\lambda_{\Omega}<1$. If $u_n$ is a positive solution of \eqref{p} for $\lambda=\lambda_n\rightarrow \infty$, then 
$u_n \rightarrow u_\mathcal{D}$ in $H^1(\Omega)$, where $u_{\mathcal{D}}$ is the unique positive solution of the Dirichlet logistic problem \eqref{pDl}.
\end{prop}

\begin{proof}
First, we claim that $u_n$ is bounded in $H^1(\Omega)$. We infer from Proposition \ref{prop2.1} that $u_n<1$ in $\overline{\Omega}$.  
From the definition of $u_n$ (\eqref{def} with $(\lambda,u)=(\lambda_n, u_n)$ and $\varphi=u_n$), we deduce that 
\begin{align*} 
\int_{\Omega}|\nabla u_n|^2 = \int_{\Omega}(u_n^2-u_n^{p+1})-\lambda_n \int_{\partial\Omega}u_n^{q+1}\leq \int_{\Omega}u_n^2\leq |\Omega|,    
\end{align*}
as desired. Immediately, up to a subsequence, $u_n \rightharpoonup u_0$, and $u_n \rightarrow u_0$ in $L^{p+1}(\Omega)$ and $L^2(\partial\Omega)$. We then deduce that 
\begin{align} \label{unq+1:2} 
\int_{\partial\Omega}u_n^{q+1} = 
\frac{1}{\lambda_n}\left( -\int_{\Omega}|\nabla u_n|^2 +  \int_{\Omega}(u_n^2-u_n^{p+1})\right) 
\leq \frac{1}{\lambda_n}\int_{\Omega}u_n^2 \longrightarrow 0, 
\end{align}
which implies that $\int_{\partial\Omega}u_0^{q+1}=0$, thus, $u_0\in H^1_0(\Omega)$. From \eqref{def} with $(\lambda,u)=(\lambda_n, u_n)$, it follows that 
\begin{align*}
\int_{\nabla} \left( \nabla u_n \nabla \varphi - u_n \varphi + u_n^p \varphi \right) = 0, \quad \varphi \in H^1_0(\Omega).     
\end{align*}
Passing to the limit provides that $u_0$ is a nonnegative solution of \eqref{pDl}. 

Next, we claim that $u_0\neq 0$. Since $E(u_n)\leq0$, we infer that  
\begin{align*}
E(u_0)\leq \varliminf_{n}E(u_n)\leq \varlimsup_{n} E(u_n)\leq0.    
\end{align*}
If $u_0=0$, then it follows that $\| u_n\|\rightarrow 0$. Say $w_n=\frac{u_n}{\| u_n\|}$, and 
$\| w_n\|=1$. Then, up to a subsequence, 
$w_n \rightharpoonup w_0\geq0$, and $w_n \rightarrow w_0$ in $L^{p+1}(\Omega)$ 
and $L^2(\partial\Omega)$. Lemma \ref{lem:wlsc} (i) shows that $w_0\neq 0$. However, we observe from \eqref{unq+1:2} that 
\begin{align*}
\int_{\partial\Omega} w_n^{q+1} \leq \frac{1}{\lambda_n}\int_{\Omega} w_n^2 \| u_n \|^{1-q}\longrightarrow 0. 
\end{align*}
This implies that $w_0\in H^1_0(\Omega)$. By a similar argument as above, we deduce that $w_0$ is a nonnegative solution of the problem
\begin{align*}
\begin{cases}
-\Delta w = w & \mbox{ in } \Omega, \\
w=0 & \mbox{ on } \partial \Omega. 
\end{cases}    
\end{align*}
Since $w_0\neq 0$ from Lemma \ref{lem:wlsc} (i), we deduce that $\lambda_{\Omega}=1$, which is contradictory for the assumption. The claim follows, thus, $u_0$ is the unique positive solution $u_{\mathcal{D}}$ of \eqref{pDl}, ensured by the strong maximum principle. 

Finally, we prove that $u_n \rightarrow u_0$ in $H^1(\Omega)$. It suffices to verify that $E(u_n)\rightarrow E(u_0)$. Observing that $E(u_0)+A(u_0)=0$ and $E(u_n)\leq -A(u_n)$, we deduce that 
\begin{align*}
E(u_0)\leq \varliminf_{n}E(u_n)\leq \varlimsup_{n}E(u_n) \leq - \varlimsup_{n}A(u_n) = - A(u_0)=E(u_0), 
\end{align*}
as desired. \end{proof}

However, if $\lambda_{\Omega}>1$, then problem \eqref{p} has no positive solutions for $\lambda>0$ large enough. 
\begin{prop} \label{prop:boundlam}
Assuming that $\lambda_{\Omega}>1$, 
\begin{align*}
\sup\left\{ \lambda>0 : \text{\eqref{p} has at least one positive solution} \right\}<\infty.     
\end{align*}
\end{prop}

\begin{proof}
Let $u$ be a positive solution of \eqref{p} for $\lambda>0$. By the change of variables \eqref{cv}, $v=\lambda^{-\frac{1}{1-q}}u$ is a positive solution of \eqref{pv} with $\mu=\lambda^{\frac{p-1}{1-q}}$. It thus suffices to show that $\mu$ possesses an upper bound. Substituting $\varphi=1$ for \eqref{def:vwsmu}, H\"older's inequality is used to deduce that 
\[
\mu \int_{\Omega} v^p 
= \int_{\Omega} v - \int_{\partial\Omega} v^q 
\leq \int_{\Omega}v 
\leq |\Omega|^{\frac{p-1}{p}} \left( \int_{\Omega} v^p \right)^{\frac{1}{p}}.    
\]
It thus follows that 
\begin{equation} 
\label{muvbdd}
\mu \| v\|_{L^p(\Omega)}^{p-1}\leq |\Omega|^{\frac{p-1}{p}}. 
\end{equation}

We then claim that 
\begin{equation} 
\label{vbddbel}
\| v \|_{L^p(\Omega)}>C, \quad \mu \to \infty     
\end{equation}
for some $C>0$. We assume by contradiction that  $\| v_n \|_{L^p(\Omega)}\rightarrow 0$ for a positive solution $v_n$ of \eqref{pv} with $\mu=\mu_n\to \infty$. Then, up to a subsequence, we deduce that $v_n \rightarrow 0$ a.e.\ in $\Omega$. In view of \eqref{cv}, we may infer that 
$v_n<1$ in $\overline{\Omega}$ because $v_n=\mu_n^{-\frac{1}{p-1}}u_n$ and $u_n<1$. Therefore, the Lebesgue dominated convergence theorem ensures that $\| v_n \|_{L^2(\Omega)} \rightarrow 0$. It follows that 
\[
\int_{\Omega} |\nabla v_n|^2 = \int_{\Omega} v_n^2 - \mu_n \int_{\Omega} v_n^{p+1} - \int_{\partial\Omega}v_n^{q+1}\leq \int_{\Omega} v_n^2 \longrightarrow 0, 
\]
that is, $\| v_n\|\rightarrow 0$. 
Say $w_n = \frac{v_n}{\| v_n\|}$, and $\| w_n\|=1$. Then, up to a subsequence, $w_n \rightharpoonup w_0\geq0$ and $w_n \rightarrow w_0$ in $L^{p+1}(\Omega)$ and $L^2(\partial\Omega)$. Lemma \ref{lem:wlsc} shows that $w_0\neq 0$ and $w_0\not\in H^1_0(\Omega)$, i.e., $\int_{\partial\Omega}w_0^{q+1}>0$. However, we find that
\[
\int_{\partial\Omega}v_n^{q+1} 
= -E(v_n) - \mu_{n} \int_{\Omega} v_n^{p+1} \leq -E(v_n); 
\]
thus, 
\[
\int_{\partial\Omega}w_n^{q+1} \leq -E(w_n) \| v_n\|^{1-q} \longrightarrow 0. 
\]
This implies that $\int_{\partial\Omega}w_0^{q+1}=0$, which is a contradiction. Claim \eqref{vbddbel} is thus verified. The desired conclusion now follows from \eqref{muvbdd} and \eqref{vbddbel}. 
\end{proof}

\subsection{Asymptotic profiles of positive solutions as $\lambda\to 0^{+}$} 

The next proposition asserts that if $u_n$ is a positive solution of \eqref{p} for $\lambda=\lambda_n \rightarrow 0^{+}$, then $(\lambda_n, u_n)$ converges to $(0,0)$, except that $u_n$ is the positive solution given by Proposition \ref{prop:bif}. 

\begin{prop} \label{prop:asymp0}
Let $u_n$ be a positive solution of \eqref{p} for $\lambda=\lambda_n \rightarrow 0^{+}$ such that $u_n \neq u_{1,\lambda_n}$, where $u_{1,\lambda_n}$ is the positive solution emanating from $(\lambda,u)=(0,1)$, ensured by Proposition \ref{prop:bif}. Then, $u_n \rightarrow 0$ in $C(\overline{\Omega})$ (and consequently in $H^1(\Omega)$). 
\end{prop}

\begin{proof}
Using Proposition \ref{prop2.1} and Lemma \ref{lem:compact}, the proof is conducted on the basis of the fact that the Neumann logistic problem, i.e., \eqref{p} with $\lambda=0$ has exactly two nonnegative solutions $u\equiv 0, 1$. \end{proof}

By employing the change of variables \eqref{cv} for \eqref{p}, we prepare some results concerning the positive solutions of \eqref{pv} for $\mu=0$ and $\mu>0$ small, which play a crucial role in characterizing the asymptotic profile of a positive solution $u_n$ of \eqref{p} for $\lambda=\lambda_n>0$ satisfying that $(\lambda, u_n)\rightarrow (0,0)$ in $\mathbb{R}\times H^1(\Omega)$. 


We prove the following {\it three} lemmas: 
\begin{lem} \label{lem:vbdd}
There exists $C>0$ such that if $u_n$ is a positive solution of \eqref{p} for $\lambda=\lambda_n>0$ such that $(\lambda_n, u_n) \rightarrow (0,0)$ in $\mathbb{R}\times H^1(\Omega)$, then $\| v_n\|\leq C$ for $v_n=\lambda_{n}^{-\frac{1}{1-q}}u_n$.  
\end{lem}

\begin{proof}
Let $u_n$ be a positive solution of \eqref{p} for $\lambda=\lambda_n\to 0^{+}$ such that $\| u_n\| \rightarrow 0$. Assume by contradiction that $\| v_n\|\rightarrow \infty$ for $v_n=\lambda_n^{-\frac{1}{1-q}}u_n$. Set $w_n:=\frac{v_n}{\| v_n\|}$, and $\| w_n\|=1$. Then, up to a subsequence, $w_n\rightharpoonup w_0\geq 0$, and $w_n \rightarrow w_0$ in $L^{p+1}(\Omega)$ and $L^2(\partial\Omega)$. By Lemma \ref{lem:wlsc} (i), we obtain $w_0\neq 0$.  

Substituting $\varphi=1$ for \eqref{def:vwsmu} with $(\mu,v)=(\mu_n,v_n)=(\lambda_n^{\frac{p-1}{1-q}},\lambda_n^{-\frac{1}{1-q}}u_n)$, it follows that 
\begin{align*}
\int_{\Omega}v_n = \int_{\Omega}u_n^{p-1}v_n + \int_{\partial\Omega}v_n^q, 
\end{align*}
thus, 
\begin{align*}
\int_{\Omega}w_n = \int_{\Omega}u_n^{p-1}w_n + \int_{\partial\Omega}w_n^q \| v_n\|^{-(1-q)}.  
\end{align*}
Since $\| u_n\|\rightarrow 0$ and $\| w_n\|=1$, we infer that $\int_{\Omega}u_n^{p-1}w_n \rightarrow 0$, thus, $\int_{\Omega}w_n \rightarrow 0$. Consequently, $\int_{\Omega} w_0 = 0$, and $w_0=0$ as desired.   \end{proof}

\begin{lem} \label{vlower}
Assume that $\lambda_{\Omega}\neq 1$. Then, for $\mu_{0}>0$, 
there exist $C>0$ such that $\| v\|\geq C$ for a positive solution $v$ of \eqref{pv} for $\mu\in [0,\mu_{0}]$.  
\end{lem}

\begin{proof}
Assume by contradiction that $\mu_n\geq0$ and $(\mu_n, v_n)\rightarrow (\mu_{\infty},0)$ in $\mathbb{R}\times H^1(\Omega)$ for some $\mu_{\infty}\geq 0$. 
Say $w_n=\frac{v_n}{\| v_n\|}$, and $\| w_n\|=1$. Then, up to a subsequence, $w_n\rightharpoonup w_\infty \geq0$, and $w_n \rightarrow w_\infty$ in $L^{p+1}(\Omega)$ and $L^2(\partial\Omega)$. By Lemma \ref{lem:wlsc} (i), 
we obtain $w_\infty\neq 0$. Substituting $\varphi=1$ for \eqref{def:vwsmu} with $(\mu, v)=(\mu_n, v_n)$, it follows that 
\begin{align*}
\int_{\partial\Omega}v_n^q = \int_{\Omega}v_n-\mu_n \int_{\Omega} v_n^p \leq \int_{\Omega} v_n, 
\end{align*}
thus, 
\begin{align*}
\int_{\partial\Omega}w_n^q \leq \int_{\Omega}w_n \| v_n\|^{1-q} 
\longrightarrow 0. 
\end{align*}
This implies that $w_0\in H^1_0(\Omega)$. Back to \eqref{def:vwsmu} with $(\mu, v)=(\mu_n, v_n)$, 
we obtain that 
\begin{align*}
\int_{\Omega} \left( \nabla v_n \nabla \varphi - v_n \varphi + \mu_n v_n^p \varphi \right) = 0,  \quad \varphi\in H^1_0(\Omega), 
\end{align*}
and 
\begin{align*}
\int_{\Omega} \left( \nabla w_n \nabla \varphi - w_n \varphi + \mu_n w_n^p 
\| v_n\|^{p-1} \varphi \right) = 0. 
\end{align*}
Passing to the limit yields  
\begin{align*}
\int_{\Omega} \left( \nabla w_\infty \nabla \varphi - w_\infty \varphi \right) = 0. 
\end{align*}
Since $w_{\infty}\geq 0$ and $w_{\infty}\neq 0$, this implies that $\lambda_{\Omega}=1$, as desired.  \end{proof}

\begin{lem} \label{lem:nomu0}
Assume that $\lambda_{\Omega}<1$. Then, problem \eqref{pv} has no positive solution for $\mu=0$. 
\end{lem}

\begin{proof}
By the assumption $\lambda_{\Omega}<1$, we obtain the unique positive solution $u_{\mathcal{D}}$ of \eqref{pDl}. If $v$ is a positive solution of \eqref{pv} for $\mu=0$, then, substituting $\varphi=u_{\mathcal{D}}$ for \eqref{def:vwsmu}, 
we deduce that 
\begin{align*}
0=\int_{\Omega} \left( \nabla v \nabla u_{\mathcal{D}}- vu_{\mathcal{D}} \right) + \int_{\partial\Omega} v^q u_{\mathcal{D}} = \int_{\Omega} \left( \nabla v \nabla u_{\mathcal{D}}- vu_{\mathcal{D}} \right). 
\end{align*}
Using the fact that $v\in W^{1,r}(\Omega)$ with $r>N$, we deduce by the divergence theorem that 
\begin{align*}
\int_{\Omega}\left( u_{\mathcal{D}}-u_{\mathcal{D}}^{p} \right)v = \int_{\Omega} \left( -\Delta u_{\mathcal{D}} \right) v = \int_{\Omega} \nabla u_{\mathcal{D}} \nabla v - \int_{\partial\Omega} \frac{\partial u_{\mathcal{D}}}{\partial\nu} v.     
\end{align*}
Combining these two assertions leads us to the contradiction
\begin{align*}
0<\int_{\Omega} u_{\mathcal{D}}^{\, p}v = \int_{\partial\Omega}\frac{\partial u_{\mathcal{D}}}{\partial\nu} v <0. 
\end{align*} \end{proof}

In the case of $\lambda_{\Omega}>1$, we then establish the following asymptotic profile of a positive solution $u_n$ of \eqref{p} for $\lambda=\lambda_n>0$ such that $(\lambda_n,  u_n) \rightarrow (0,0)$ in $\mathbb{R}\times H^1(\Omega)$.

\begin{prop} \label{prop:asymp}
Assume that $\lambda_{\Omega}>1$. If $u_n$ is a positive solution of \eqref{p} for $\lambda=\lambda_n>0$ such that $(\lambda_n, u_n) \rightarrow (0,0)$ in $\mathbb{R}\times H^1(\Omega)$, then up to a subsequence, 
\begin{align*}
v_n=\lambda_n^{-\frac{1}{1-q}}u_n \longrightarrow v_0 \quad\mbox{ in } H^1(\Omega).  
\end{align*}
Here, $v_0$ is a positive solution of \eqref{lp}, which satisfies that 
$v_0>0$ on $\Gamma\subset \partial\Omega$ with the condition that $|\Gamma|>0$. 
\end{prop}

\begin{proof}
Similarly as in the proof of Proposition \ref{prop2.1} (ii), we deduce that $v>0$ on $\Gamma\subset \partial\Omega$ with  $|\Gamma|>0$ for a positive solution $v$ of \eqref{lp}. 

From Lemma \ref{lem:vbdd},  
$v_n$ is bounded in $H^1(\Omega)$. 
Immediately, up to a subsequence, $v_n \rightharpoonup v_0\geq0$, and $v_n\rightarrow v_0$ in $L^{p+1}(\Omega)$ and $L^2(\partial\Omega)$. 
From \eqref{def:vwsmu}, it follows that 
\begin{align*}
\int_{\Omega} \left( \nabla v_n \nabla \varphi - v_n \varphi + \lambda_n^{\frac{p-1}{1-q}}v_n^p \varphi \right) + \int_{\partial\Omega} v_n^{q}\varphi=0, \quad\varphi\in H^1(\Omega), 
\end{align*}
and passing to the limit deduces 
\begin{align*}
\int_{\Omega} \left(\nabla v_0 \nabla \varphi - v_0 \varphi \right) +  \int_{\partial\Omega} v_0^{q}\varphi=0. 
\end{align*}
This means that $v_0$ is a nonnegative solution of \eqref{lp} (i.e., \eqref{pv} with $\mu=0$). Moreover, combining Lemma \ref{vlower} and Lemma \ref{lem:wlsc} (i) provides that $v_0\neq 0$, thus, $v_0$ is a positive solution of \eqref{lp}. 

Finally, we prove that $v_n \rightarrow v_0$ in $H^1(\Omega)$. It suffices to show that $E(v_n) \rightarrow E(v_0)$. Observing that $E(v_n)\leq -B(v_n)$, we deduce that $E(v_0)\leq \varliminf_{n}E(v_n)\leq \varlimsup_{n} E(v_n) \leq -B(v_0)= E(v_0)$, as desired. \end{proof}

As a byproduct of Proposition \ref{prop:asymp}, we obtain the following instability result for the positive solutions of \eqref{p} as $\lambda \to 0^{+}$. 

\begin{prop} \label{prop:unsta}
Assume that $\lambda_{\Omega}>1$. Let $u_n$ be a positive solution of \eqref{p} for $\lambda=\lambda_n>0$ such that $(\lambda_n, u_n) \rightarrow (0,0)$ in $\mathbb{R}\times H^1(\Omega)$. If $u_n>0$ in $\overline{\Omega}$, then $u_n$ is {\rm unstable} for $n$ large enough.   
\end{prop}

\begin{proof}
Let $u_n$ be a positive solution of \eqref{p} for $\lambda=\lambda_n>0$ such that $(\lambda_n,u_n) \rightarrow (0,0)$ in $\mathbb{R}\times H^1(\Omega)$, and $u_n > 0$ in $\overline{\Omega}$. Considering the linearized eigenvalue problem \eqref{linearp} at $(\lambda,u)=(\lambda_n, u_n)$, we prove that the smallest eigenvalue $\gamma_{1,n}$ is negative for $n$ large enough. 

Recall that $f(t)=t-t^p$ and $g(t)=t^q$ for $t>0$. Using a positive eigenfunction $\varphi_{1,n}>0$ in $\overline{\Omega}$ associated with $\gamma_{1,n}$, we deduce that 
\begin{align*}
-\left( \frac{u_n}{\varphi_{1,n}} \right)\sum_{j=1}^N \frac{\partial}{\partial x_j} \left( \varphi_{1,n}^2 \frac{\partial}{\partial x_j}\left( \frac{u_n}{\varphi_{1,n}}\right) \right) 
= u_n \left( f(u_n) - f'(u_n)u_n \right) - \gamma_{1,n}u_n^2. 
\end{align*}
Both sides are integrated over $\Omega$, and we use the divergence theorem to deduce that 
\begin{align*}
\gamma_{1,n} \left( \int_{\Omega} u_n^2 + \int_{\partial\Omega} u_n^2 \right)
&=-\int_{\Omega} \varphi_{1,n}^2 \left\vert \nabla \left( \frac{u_n}{\varphi_{1,n}}\right)\right\vert^2 - \int_{\Omega} u_n^{3} F^{\prime}(u_n) + \lambda_n \int_{\partial\Omega} u_n^{3} G^{\prime}(u_n) \\ 
& \leq - \int_{\Omega} u_n^{3} F^{\prime}(u_n) + \lambda_n \int_{\partial\Omega} u_n^{3} G^{\prime}(u_n) =:I_n, 
\end{align*}
where $F(t) = \frac{f(t)}{t}$ and $G(t) = \frac{g(t)}{t}$. Once we verify that 
up to a subsequence, $I_n<0$, the proof is completed. 

Since $F'(t)=-(p-1)t^{p-2}$ and $G'(t)=(q-1)t^{q-2}$, it follows that 
\begin{align*}
I_n 
&= (p-1)\int_{\Omega} u_n^{p+1} -(1-q)\lambda_n \int_{\partial\Omega} u_n^{q+1} \\ 
&= (p-1)\int_{\Omega} \lambda_n^{\frac{p+1}{1-q}}v_n^{p+1} -(1-q)\lambda_n^{\frac{2}{1-q}}  \int_{\partial\Omega} v_n^{q+1} \\ 
&=\lambda_n^{\frac{2}{1-q}}\left\{ (p-1)\lambda_n^{\frac{p-1}{1-q}}\int_{\Omega} v_n^{p+1} - (1-q) \int_{\partial\Omega} v_n^{q+1} \right\} =: \lambda_n^{\frac{2}{1-q}} \hat{I}_n, 
\end{align*}
where $v_n = \lambda_n^{-\frac{1}{1-q}}u_n$. Proposition \ref{prop:asymp} enables us to deduce that up to a subsequence,  $\hat{I}_n \rightarrow -(1-q)\int_{\partial\Omega}v_0^{q+1}<0$, as desired.  \end{proof}



In the case of $\lambda_{\Omega}<1$, there is no positive solution $u_n$ of \eqref{p} for $\lambda=\lambda_n>0$ such that $(\lambda_n,u_n) \rightarrow (0,0)$ in $\mathbb{R}\times H^1(\Omega)$. 
\begin{prop} \label{prop:lbdd}
Assume that $\lambda_{\Omega}<1$. Then, there exists $C_0>0$ and $\lambda_0>0$  such that $\| u\|\geq C_0$ for a positive solution $u$ of \eqref{p} with $\lambda\in (0,\lambda_0)$. 
\end{prop}

\begin{proof} 
Assume by contradiction that $(\lambda_n,u_n)\rightarrow (0,0)$ in $\mathbb{R}\times H^1(\Omega)$ for a positive solution $u_n$ of \eqref{p} for $\lambda=\lambda_n>0$. 
Using $\mu_n = \lambda_n^{\frac{p-1}{1-q}}$ and $v_n = \lambda_{n}^{-\frac{1}{1-q}}u_n$ from \eqref{cv}, $(\mu,v)=(\mu_n,v_n)$ is a positive solution of \eqref{pv}, and Lemma \ref{lem:vbdd} shows that $v_n$ is bounded in $H^1(\Omega)$. It follows that up to a subsequence, $v_n \rightharpoonup v_0\geq 0$, and $v_n \rightarrow v_0$ in $L^{p+1}(\Omega)$ and $L^2(\partial\Omega)$. Combining Lemma \ref{lem:wlsc} (i) and Lemma \ref{vlower} deduces that $v_0\neq 0$. 

Observing that $(\mu,v)=(\mu_n,v_n)$ admits \eqref{def:vwsmu}, we deduce by 
passing to the limit that 
\begin{align*}
\int_{\Omega} \left( \nabla v_0 \nabla \varphi - v_0\varphi \right) + \int_{\partial\Omega} v_0^q \varphi = 0, \quad \varphi \in H^1(\Omega),     
\end{align*}
which implies that $v_0$ is a nonnegative solution of \eqref{pv} with $\mu=0$ (i.e., \eqref{lp}). Lemma \ref{lem:nomu0} shows that $v_0=0$, as desired.  \end{proof}

As a consequence of Proposition \ref{prop:lbdd}, we have the uniqueness of a positive solution of \eqref{p} for $\lambda>0$ small. 
\begin{cor} \label{cor:uniq}
Assume that $\lambda_{\Omega}<1$. Then, a positive solution of \eqref{p} is unique for $\lambda>0$ small, which is given by $u_{1,\lambda}$ of Proposition \ref{prop:bif}. 
\end{cor}

\begin{proof}
Let $u_n$ be a positive solution of \eqref{p} for $\lambda=\lambda_n \rightarrow 0^{+}$. Lemma \ref{lem:compact} then shows that up to a subsequence,  
$u_n \rightharpoonup u_0\geq0$, and $u_n \rightarrow u_0$ in $C(\overline{\Omega})$. From the condition that 
\begin{align*}
\int_{\Omega} \left( \nabla u_n \nabla \varphi -u_n \varphi +  u_n^p \varphi \right) +\lambda_n \int_{\partial\Omega} u_n^{q} \varphi = 0, \quad \varphi \in H^1(\Omega),  
\end{align*} 
we deduce by passing to the limit that 
\begin{align*}
\int_{\Omega} \left( \nabla u_0 \nabla \varphi - u_0 \varphi +  u_0^p \varphi \right) = 0. 
\end{align*}
This implies that $u_0$ is a nonnegative solution of \eqref{p} with $\lambda=0$, thus, $u_0\equiv 0$ or $1$. 

Proposition \ref{prop:lbdd} shows that 
$\| u_n\|\geq C_{0}$, and Lemma \ref{lem:wlsc} (i) yields that $u_0\neq 0$. Thus, 
$u_n \rightarrow 1$ in $C(\overline{\Omega})$. 
Using a bootstrap argument and the compactness result, it follows that up to a subsequence, $u_n \rightarrow 1$ in $C^{2+\beta}(\overline{\Omega})$ for some $\beta \in (0,1)$. 
Therefore, Proposition \ref{prop:bif} shows that $u_n = u_{1,\lambda_n}$ for $n$ large, thus, the desired conclusion follows. \end{proof}


We then prove Theorems \ref{th2} and \ref{th3}. 


\begin{proof}[Proof of Theorem \ref{th2}] 
The existence result for a positive solution of \eqref{p} is due to Proposition \ref{prop:subsuper}. 
Assertions (i) and (ii) follow from Corollary \ref{cor:uniq} and Proposition \ref{prop:asympinfty}, respectively.  \end{proof}


\begin{proof}[Proof of Theorem \ref{th3}]
The multiplicity result for the positive solutions follows from Propositions \ref{prop:min+} and \ref{prop:minJlam-}. We choose the minimizers $u_{\lambda}^{\pm}$ as $u_{\lambda}^{\pm}\geq0$ and $u_{\lambda}^{\pm}\neq 0$. Therefore, by \cite[Theorem 2.3]{BZ03}, we deduce that $u_{\lambda}^{\pm}$ is usual critical points for $J_{\lambda}$, and $u_{\lambda}^{\pm}$ is nontrivial, nonnegative weak solutions of \eqref{p}. This implies that $u_{\lambda}^{\pm}$ are positive solutions of \eqref{p}, as stated in the Introduction. 
The behavior of $U_{1,\lambda}(=u_{\lambda}^{+})$ as $\lambda\to 0^{+}$ is verified in a similar manner as in the proof of Corollary \ref{cor:uniq} (by \eqref{ulaminftbdd}, we see that $u_{\lambda}^{+}$ admits the assertion of Proposition \ref{prop:lbdd}). 
The assertion that $U_{2,\lambda}(=u_{\lambda}^{-})$ converges to $0$ in $H^1(\Omega)$ as $\lambda \to 0^{+}$ is verified by \eqref{ulamminu:asympt}. 
The nonexistence result for positive solutions comes from Proposition \ref{prop:boundlam}. 
Finally, assertions (i) and (ii) follow from Propositions \ref{prop:asymp} and \ref{prop:unsta}, respectively.  \end{proof}


\section{Existence of bounded, closed, and connected subsets of positive solutions} 

\label{sec:topol}

This section is devoted to the proof of Theorem \ref{th4}.  
First of all, we consider the existence of a principal eigenvalue of the Steklov eigenvalue problem
\begin{align} \label{eprosig}
\begin{cases}
-\Delta \psi = \psi & \mbox{ in } \Omega, \\
\frac{\partial \psi}{\partial\nu}=\sigma \psi & \mbox{ on } \partial\Omega, 
\end{cases}    
\end{align}
where $\sigma \in \mathbb{R}$ is an eigenvalue parameter. We recall that a principal eigenvalue of \eqref{eprosig} is referred to as an eigenvalue with constant sign eigenfunctions. 
A nonnegative eigenfunction of \eqref{eprosig} is positive in $\overline{\Omega}$ by the strong maximum principle and boundary point lemma.

For \eqref{eprosig}, we present the following result that comes from \cite[Lemma 9]{GM09}. 

\begin{lem} \label{lem:sig1}
If \eqref{eprosig} has a principal eigenvalue $\sigma_1$, then $\sigma_1<0$ and $\lambda_{\Omega}>1$. Conversely, if $\lambda_{\Omega}>1$, then  \eqref{eprosig} possesses a unique principal eigenvalue $\sigma_1<0$, which is characterized by the variational formula 
\begin{align*} 
\sigma_1 =\inf\left\{ E(\psi) : \psi \in H^1(\Omega), \ \int_{\partial\Omega} \psi^2 = 1 \right\}.  
\end{align*}
Additionally, the infimum is attained, and $\sigma_1$ is simple.  
    
\end{lem}



Then, we discuss bifurcation from $\{(\lambda,0)\}$ for the positive solutions of \eqref{p}. However, this is a non standard bifurcation problem in the sense that Crandall and Rabinowitz' local bifurcation theory from simple eigenvalues \cite{CR71} is not directly applicable, because the function $t^q$ ($t\geq0$) with $0<q<1$ is not right differentiable at $t=0$. 
To overcome this difficulty, we consider a {\it regularization} for \eqref{p} as follows. 
\begin{align} \label{rp}
\begin{cases}
-\Delta u = u(1-|u|^{p-1}) & \mbox{ in } \Omega, \\
\frac{\partial u}{\partial \nu}  = -\lambda (u+\varepsilon)^{q-1}u 
\ \ \left( =-\lambda\left( \frac{u}{u+\varepsilon} \right)^{1-q}u^q\right) & \mbox{ on } \partial\Omega, 
\end{cases}    
\end{align}
with a regularization parameter $\varepsilon\in (0,1)$. Formally, we regard \eqref{p} as $\eqref{rp}$ with $\varepsilon=0$ for nonnegative solutions. This formal  observation will be justified by a topological argument by Whyburn \cite{Wh64}.  

First, we evaluate bifurcation from $\{(\lambda,0):\lambda>0\}$ for the positive solutions of \eqref{rp} with a fixed $\varepsilon\in (0,1)$ and second, how the bifurcating positive solution set $\{ (\lambda,u)\}$ behaves as $\varepsilon \rightarrow 0^{+}$. We remark that a positive solution of \eqref{rp} is positive in $\overline{\Omega}$ by the strong maximum principle and boundary point lemma because \eqref{rp} is a regular problem. For the bifurcation analysis for the positive solutions of \eqref{rp}, we consider the linearized eigenvalue problem 
\begin{align}  \label{p6.5}
\begin{cases}
-\Delta \varphi = (1-p|u|^{p-1}) \varphi + \gamma \varphi & \mbox{ in } \Omega, \\
\frac{\partial \varphi}{\partial \nu}  = -\lambda \left\{ 
(q-1)(u+\varepsilon)^{q-2}u + (u+\varepsilon)^{q-1} \right\}\varphi 
+ \gamma \varphi  & \mbox{ on } \partial\Omega. 
\end{cases}    
\end{align}
Substitute $u=0$ for \eqref{p6.5}, and consider the case $\gamma=0$:
\begin{align} \label{r-lep}
\begin{cases}
-\Delta \varphi = \varphi  & \mbox{ in } \Omega, \\
\frac{\partial \varphi}{\partial \nu}  = -\lambda \varepsilon^{q-1} \varphi & \mbox{ on } \partial\Omega. 
\end{cases}    
\end{align}
If $\lambda_{\Omega}>1$, then Lemma \ref{lem:sig1} ensures that problem \eqref{r-lep} has a unique principal eigenvalue $\lambda_{\varepsilon}>0$, which is simple and satisfies that  
\begin{align} \label{lamep} 
\lambda_{\varepsilon}=\lambda_1 \varepsilon^{1-q} \quad (\mbox{implying that $\lambda_{\varepsilon} \rightarrow 0$ as $\varepsilon \to 0^{+}$}). 
\end{align}  
Before stating our bifurcation result for \eqref{rp}, we establish several {\it a priori} bounds for the positive solutions of \eqref{rp}. 

\begin{prop} \label{prop:boundep}
The following assertions hold:
\begin{enumerate} \setlength{\itemsep}{0.2cm} 
\item If $u$ is a positive solution of \eqref{rp} for $\lambda>0$, then $u<1$ in $\overline{\Omega}$. 
\item Assume that $\lambda_{\Omega}>1$. Then, we have the following:
\begin{enumerate}  \setlength{\itemsep}{0.2cm} 
\item There exists $\Lambda_0>0$ such that if problem \eqref{rp} has a positive solution for $\lambda>0$, then $\lambda \leq \Lambda_0$, where $\Lambda_0$ does not depend on $\varepsilon \in [0,1)$. 

\item For each $\delta>0$ that is small, there exists $C=C_{\varepsilon,\delta}>0$ such that $\| u\|_{C(\overline{\Omega})}\geq C$ for a positive solution of \eqref{rp} with $\lambda\in[0,\lambda_{\varepsilon}-\delta]\cup [\lambda_{\varepsilon}+\delta, \Lambda_0]$.
\end{enumerate}

\item For $0<\delta_1<\delta_2<1$, there exists $\Lambda_1>0$ such that if problem \eqref{rp} has a positive solution $u$ of \eqref{rp} such that $\delta_1<\| u\|_{C(\overline{\Omega})}<\delta_2$, then $\lambda\geq \Lambda_1$. Here, $\Lambda_1$ does not depend on $\varepsilon \in [0,1)$. 
\end{enumerate}
\end{prop}

\begin{proof}
(i) The proof is carried out in the same spirit of that for Proposition \ref{prop2.1} (i). Indeed, it suffices to notice that a constant $u=c\geq1$ is a supersolution of \eqref{rp}. 

(ii-a) 
We assume by contradiction that problem \eqref{rp} has a positive solution $u_n$ for $\lambda=\lambda_n \rightarrow \infty$ and $\varepsilon = \varepsilon_n \in [0,1)$. Then, it follows that 
\begin{align} \label{epn2046}
\int_{\partial\Omega} \frac{u_n^2}{(u_n + \varepsilon_n)^{1-q}} 
= \frac{1}{\lambda_n} \left\{ \int_\Omega \biggl( -|\nabla u_n|^2 + u_n^2 \biggr) - \int_\Omega u_n^{p+1} \right\} \leq \frac{1}{\lambda_n}\int_\Omega u_n^2. 
\end{align}
Say $w_n = \frac{u_n}{\| u_n\|}$, and then, up to a subsequence, 
$w_n \rightharpoonup w_\infty\geq 0$, $w_n \to w_\infty$ in $L^{p+1}(\Omega)$ and $L^2(\partial\Omega)$ for some $w_\infty \in H^1(\Omega)$. Since $u_n<1$ and $\varepsilon_n < 1$, we infer from \eqref{epn2046} that 
\begin{align*} 
\frac{1}{2^{1-q}}\int_{\partial\Omega} w_n^2 \leq \frac{1}{\lambda_n}\int_\Omega w_n^2 \longrightarrow 0. 
\end{align*}
This implies that $\int_{\partial\Omega}w_\infty^2 = 0$, i.e., $w_\infty \in H^1_0(\Omega)$. The rest of the proof is in the same manner as that for Proposition \ref{prop:boundlam}. 

(ii-b) This is a direct consequence of \cite[Proposition 18.1]{Am76}. 

(iii) Assume by contradiction that for some $0<\delta_1<\delta_2<1$, $u_n$ is a positive solution of \eqref{rp} with $\varepsilon_n\in [0,1)$ for $\lambda=\lambda_n \rightarrow 0^{+}$ such that $\delta_1<\| u_n\|_{C(\overline{\Omega})}<\delta_2$. 
It follows that $u_n$ is bounded in $H^1(\Omega)$. Immediately, up to a subsequence, $u_n \rightharpoonup u_0$, $u_n \rightarrow u_0$ in $L^{p+1}(\Omega)$ 
and $L^2(\partial\Omega)$, $u_n \rightarrow u_0$ a.e., and $\varepsilon_n \rightarrow \varepsilon_0\in [0,1]$. Additionally, 
$u_n \rightarrow u_0\geq0$ in $C(\overline{\Omega})$ as in the proof of Lemma \ref{lem:compact}. We deduce that 
\begin{align} \label{ukepk}
\frac{u_n}{u_n + \varepsilon_{n}}\leq 1 \quad\mbox{ in } \overline{\Omega}, 
\end{align}
and so we can prove that $u_n$ is bounded in $W^{1,r}(\Omega)$ for $r>N$, following the bootstrap argument developed in the proof of \cite[Theorem 2.2]{Ro2005}. Consequently, $\delta_1\leq \| u_0  \|_{C(\overline{\Omega})}\leq \delta_2$ and particularly, $u_0\not\equiv 0, 1$. 

Since $u_n$ is a positive solution of \eqref{rp} with $(\lambda, \varepsilon)=(\lambda_n, \varepsilon_n)$, we observe 
\begin{align*}
\int_{\Omega} \left( \nabla u_n \nabla \varphi - u_n\varphi + u_n^p \varphi \right) + \lambda_n\int_{\partial\Omega} (u_n + \varepsilon_n)^{q-1}u_n \varphi = 0, \quad \varphi \in H^1(\Omega). 
\end{align*}
Passing to the limit shows that 
\begin{align*}
\int_{\Omega} \left( \nabla u_0 \nabla \varphi - u_0\varphi + u_0^p \varphi \right) = 0, 
\end{align*}
because 
\begin{align*}
(u_n + \varepsilon_n)^{q-1}u_n = \left( \frac{u_n}{u_n + \varepsilon_n}\right)^{1-q}u_n^q \leq \delta_{2}^{\, q} \quad\mbox{ in } \overline{\Omega}. 
\end{align*}
This means that $u_0$ is a nonnegative solution of the Neumann logistic problem \eqref{Neup}. Therefore, $u_0\equiv 0$ or $1$, as desired. \end{proof}

Using Proposition \ref{prop:boundep}, we prove the following existence result for bifurcating positive solutions of \eqref{rp} at $(\lambda,u)=(\lambda_{\varepsilon},0)$. 
Taking into account that problem \eqref{rp} is regular near $(\lambda_\varepsilon,0)$, the proof of Proposition \ref{prop:rp:bif} is carried out in the same spirit of that for \cite[Proposition 2.2]{Um2013} (see also \cite[Theorem 1.1]{Um2010}), 
which is thus omitted here. 

\begin{prop} \label{prop:rp:bif}
Suppose that $\lambda_{\Omega}>1$. Let $\varepsilon \in (0,1)$ be fixed. Then, problem \eqref{rp} possesses a bounded {\rm component} (i.e., maximal, closed, and connected subset) $\mathcal{C}_\varepsilon = \{ (\lambda,u )\}$ of nonnegative solutions in $[0,\infty)\times C(\overline{\Omega})$ for some $\theta\in (0,1)$ that bifurcates from $\{(\lambda,0)\}$ at $(\lambda_\varepsilon,0)$. Moreover, it holds that 
\begin{enumerate} \setlength{\itemsep}{0.2cm} 
\item $(0,1), (\lambda_\varepsilon,0)\in \mathcal{C}_\varepsilon$; 
\item $\mathcal{C}_\varepsilon$ does not meet any $(0,u)$ or $(\lambda,0)$ except for $(0,1), (\lambda_\varepsilon,0)$;
\item $\mathcal{C}_\varepsilon\setminus \{(\lambda_\varepsilon,0)\}$ is contained in the positive solution set of \eqref{rp}, 
\end{enumerate}
see Figure \ref{figep}. 
\end{prop}

    \begin{figure}[!htb]
    \centering 
        \includegraphics[scale=0.19]{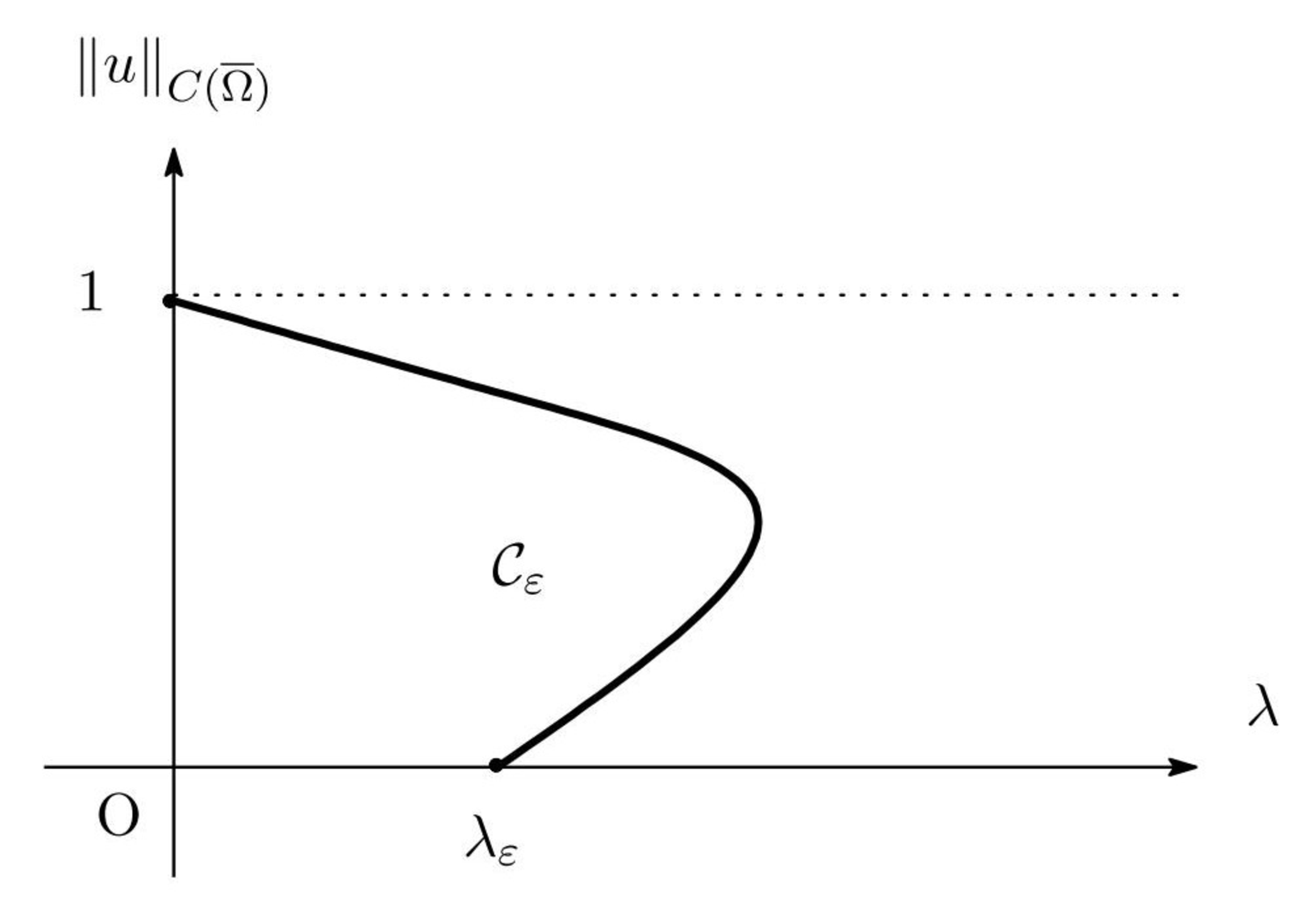}
		  \caption{Component $\mathcal{C}_{\varepsilon}$. 
		  }
		\label{figep} 
    \end{figure}

The following nonexistence result is derived from Lemma \ref{vlower} by considering \eqref{cv}, which plays an important role in determining the limiting behavior of $\mathcal{C}_{\varepsilon}$ as $\varepsilon \to 0^{+}$. 


\begin{prop} \label{prop:nobif}
Assume $\lambda_{\Omega}\neq 1$. Let $\lambda_0>1$. 
Then, there exists $C>0$ such that $\| u\|\geq C$ for a positive solution $u$ of \eqref{p} with $\lambda\in [1/\lambda_0, \lambda_0]$, meaning that there is no bifurcation point on $\{ (\lambda,0) : \lambda>0 \}$ for positive solutions of \eqref{p}.  
\end{prop} 

We then end the proof of Theorem \ref{th4}. 
\begin{proof}[End of proof of Theorem \ref{th4}]
Let $X$ be a metric space, and $\mathcal{E}_n \subset X$. Set 
\begin{align*}
& \varliminf_{n\to \infty} \mathcal{E}_n := \{ x \in X : \lim_{n\to \infty}{\rm dist}\, (x, \mathcal{E}_n) = 0 \}, \\
& \varlimsup_{n\to \infty} \mathcal{E}_n := \{ x \in X : \varliminf_{n\to \infty}{\rm dist}\, (x, \mathcal{E}_n) = 0 \}.
\end{align*}
We then obtain the following (\cite[(9.12) Theorem]{Wh64}): 
\begin{theorem} \label{thm:W}
Assume that $\left\{ \mathcal{E}_{n}\right\}$ is a sequence of connected sets which satisfies that 
\begin{enumerate} \setlength{\itemsep}{0.2cm} 
\item[(i)] $\displaystyle \bigcup_{n\geq 1}\mathcal{E}_n$ is precompact; 
\item[(ii)] $\displaystyle \varliminf_{n\to \infty}\mathcal{E}_n \neq \emptyset$. 
\end{enumerate}
Then, $\displaystyle \varlimsup_{n\to \infty}\mathcal{E}_n$ is nonempty, closed and connected. 
\end{theorem}

We then use Theorem \ref{thm:W} to obtain a closed and connected limit set of $\mathcal{C}_{\varepsilon}$ as $\varepsilon \to 0^{+}$, as Theorem \ref{th4} demands. We introduce the metric space $X := \mathbb{R}\times C(\overline{\Omega})$ with the metric function given by 
\[
d((\lambda, u), (\mu, v)) := |\lambda - \mu| + \| u - v \|_{C(\overline{\Omega})}
\quad \mbox{for} \ \ (\lambda, u), (\mu, v) \in \mathbb{R}\times C(\overline{\Omega}).
\]
Let $\varepsilon_n \in (0,1)$ be such that $\varepsilon_n \to 0^{+}$, and $\mathcal{C}_n :=\mathcal{C}_{\varepsilon_n}$ the bounded component ensured by Proposition \ref{prop:rp:bif}. 
From Proposition \ref{prop:boundep} (i) and (ii-a), we then deduce that 
\begin{align} \label{Cep:bdd}
\bigcup_{n\geq1} \mathcal{C}_n \subset \{ (\lambda, u) \in \mathbb{R}\times C(\overline{\Omega}) : 0\leq \lambda \leq \Lambda_{0}, \ u\leq 1 \ \mbox{ in } \overline{\Omega} \}. 
\end{align}
Additionally, from \eqref{lamep}, we also deduce that 
\begin{align} \label{liminfnon}
(0,0), (0,1) \in \varliminf_{n\to\infty} \mathcal{C}_n. 
\end{align}
We claim that 
\begin{align} \label{precom}
\bigcup_{n\geq1} \mathcal{C}_n \ \ \mbox{ is precompact. }
\end{align}
For $\{ (\lambda_k, u_k) \}_{k=1}^\infty \subset \bigcup_{n\geq1} \mathcal{C}_n$, 
we deduce that $(\lambda_k, u_k) \in \mathcal{C}_{n_k}$ for some $n_k$. Note that 
\begin{align} \label{uksol}
\int_{\Omega} \left( \nabla u_k \nabla \varphi -u_k \varphi + u_k^p \varphi \right) + \lambda_k \int_{\partial\Omega} \left( \frac{u_k}{u_k +\varepsilon_{n_k} }\right)^{1-q}u_k^q \varphi=0, \quad \varphi \in H^1(\Omega).     
\end{align}
From \eqref{Cep:bdd}, we may infer that $\{\lambda_k\}$ is a convergent sequence. Moreover, since  $u_k$ is bounded in $C(\overline{\Omega})$, it is bounded in $H^1(\Omega)$. Indeed, \eqref{uksol} with $\varphi=u_k$ implies that 
\begin{align*}
\int_{\Omega} |\nabla u_k|^2 = \int_{\Omega} \left( u_k^2 - u_k^{p+1} \right) - \lambda_k \int_{\partial\Omega} \left( \frac{u_k}{u_k +\varepsilon_{n_k} }\right)^{1-q} 
u_{k}^{q+1}\leq C.    
\end{align*}
Hence, $u_k$ has a convergent subsequence in $C(\overline{\Omega})$, which is deduced in the same argument as in proof of Proposition \ref{prop:boundep}(iii).  Claim \eqref{precom} is thus verified. 

Assertions \eqref{liminfnon} and \eqref{precom} then enable us to apply Theorem \ref{thm:W}, deducing that $\mathcal{C}_{0} := \varlimsup_{n\to\infty}\mathcal{C}_n$ is nonempty, closed, and connected in $[0,\infty)\times C(\overline{\Omega})$ such that $(0,0), (0,1)\in \mathcal{C}_0$.  
From \eqref{Cep:bdd}, $\mathcal{C}_{0}$ is bounded in $[0,\infty)\times C(\overline{\Omega})$. We claim that $\mathcal{C}_{0}$ consists of nonnegative solutions of \eqref{p}. Given $(\lambda, u) \in \mathcal{C}_{0}$, there exists $(\lambda_k, u_k) \in \mathcal{C}_{n_k}$ such that $\varepsilon_{n_k} \to 0^+$ and $(\lambda_k, u_k) \to (\lambda, u)$ in $\mathbb{R} \times C(\overline{\Omega})$, deducing that $u\geq0$, It follows that $(\lambda_k, u_k)$ satisfies \eqref{uksol}, and $u_k$ is bounded in $C(\overline{\Omega})$ and $H^1(\Omega)$. 
By the same argument above, we infer that up to a subsequence, $u_k \rightharpoonup u$. We observe that 
\begin{align*}
&\left( \frac{u_k}{u_k +\varepsilon_{n_k} }\right)^{1-q}u_k^q \leq u_k^q \longrightarrow 0 \quad \mbox{ for $x\in\partial\Omega$ with $u(x)=0$}, \\
& \left( \frac{u_k}{u_k +\varepsilon_{n_k} }\right)^{1-q}u_k^q \longrightarrow u^q \quad 
\mbox{ for $x\in\partial\Omega$ with $u(x)>0$}, 
\end{align*}
thus 
\begin{align*}
\left( \frac{u_k}{u_k +\varepsilon_{n_k} }\right)^{1-q}u_k^q 
\longrightarrow u^q \quad \mbox{ on } \partial\Omega. 
\end{align*}
From \eqref{ukepk}, $\left( \frac{u_k}{u_k +\varepsilon_{n_k}} \right)^{1-q}u_k^q$ is bounded in $C(\overline{\Omega})$, and passing to the limit yields 
 \begin{align*}
\int_{\Omega} \left( \nabla u \nabla \varphi - u \varphi + u^{p} \varphi \right) + \lambda  \int_{\partial\Omega} u^q \varphi=0, \quad \varphi \in H^1(\Omega), 
\end{align*}
by the Lebesgue dominated convergence theorem, as desired. 


Finally, we verify that $\mathcal{C}_0$ is the desired subcontinuum. Assertion (i) follows from \eqref{liminfnon}. Assertion (ii) follows from Proposition \ref{prop:boundep} (iii) and Proposition \ref{prop:nobif}. Assertion (iii) follows from the combination of assertion (i) with Proposition \ref{prop:bif}. The proof of Theorem \ref{th4} is now complete.  \end{proof} 




\end{document}